\documentclass[11pt]{article}
\usepackage{etex}
\usepackage{amsmath,amssymb,amsfonts,amsthm}
\usepackage{a4wide}
\usepackage{color}
\usepackage{verbatim}
\usepackage{epsfig,subfigure,epstopdf}
\usepackage[]{graphics}
\usepackage{graphicx,inputenc}
\usepackage{subfigure}
\usepackage[justification=centering]{caption}
\usepackage{pictex}
\usepackage{wrapfig}
\usepackage{multirow}
\usepackage[T1]{fontenc}
\usepackage{authblk}
\usepackage[textwidth=6.0in,textheight=9in]{geometry}

\usepackage[colorlinks,linktocpage,linkcolor=blue]{hyperref}
\usepackage{fancyhdr}

\usepackage{algorithm}
\usepackage{algorithmic}
\usepackage{enumerate}

\newtheoremstyle{exampstyle}
  {\topsep} 
  {\topsep} 
  {\itshape} 
  {} 
  {\bfseries} 
  {.} 
  {.5em} 
  {} 

\theoremstyle{exampstyle}

\numberwithin{equation}{section}
\newtheorem{theorem}{Theorem}
\newtheorem{lemma}{Lemma}[section]
\newtheorem{assumption}{Assumption}[section]

\allowdisplaybreaks
\usepackage{parskip}
\let\oldref\ref
\renewcommand{\ref}[1]{(\oldref{#1})}  
\renewcommand{\eqref}[1]{(\oldref{#1})}

\newbox\boxaddrone \newbox\boxaddrtwo

\def\N+{n\in\mathbb{N}^{+}}

\def\n{\partial{\overrightarrow{\bf n}}}

\def\1d{\mathcal{D}((-\Delta)^{\gamma+1})}

\begin{document}
\title{\large\textbf{A monotone iterative reconstruction method for an inverse drift problem in a two-dimensional parabolic equation}}
\author[1]{Liuying Zhang\thanks{zhangly8@sustech.edu.cn}}
\author[1]{Wenlong Zhang\thanks{zhangwl@sustech.edu.cn}}
\author[2,3]{Zhidong Zhang\thanks{zhangzhidong@mail.sysu.edu.cn}}
\affil[1]{\small{Department of Mathematics, Southern University of Science and Technology, Shenzhen, 518055, China}}
\affil[2]{\small{School of Mathematics (Zhuhai), Sun Yat-sen University, Zhuhai 519082, Guangdong, China}}
\affil[3]{\small{Guangdong Province Key Laboratory of Computational Science, Sun Yat-sen
University, Guangzhou 510000, Guangdong, China}}

\maketitle

\begin{abstract}
\noindent 
We study an inverse drift problem for a two-dimensional parabolic equation on
the unit square with mixed boundary conditions, where the drift coefficient
is recovered from terminal observation data $g=u(\cdot,T)$. 
A monotone operator is constructed whose fixed point coincides with the unknown drift, 
yielding uniqueness in an admissible class and a constructive iterative reconstruction scheme.
Numerical experiments illustrate the monotone convergence and the effectiveness of the proposed method, and show that it remains effective for noisy terminal data under the denoising strategy.\\


\end{abstract}


\section{Introduction.}
\subsection{Mathematical statement.}
We consider an inverse drift problem associated with the following parabolic equation:
\begin{equation}\label{PDE}
 \begin{cases}
  \begin{aligned}
   (\partial_t-\Delta+q(x)\partial_{x_1}+C_p)u(x,t)&=f(x), &&(x,t)\in \Omega\times(0,T];\\
   \mathcal Bu(x,t)&=b(x,t),&&(x,t)\in\partial\Omega\times(0,T];\\
   u(x,0)&=u_0(x),&&x\in\Omega.\\
  \end{aligned}
 \end{cases}
\end{equation}
The domain $\Omega$ is the unit square $(0,1)\times(0,1)$ in $\mathbb R^2$ and $x\in\Omega$ could be denoted by $(x_1,x_2)$.  
For the boundary $\partial\Omega$, we denote  
$$l_1=[0,1]\times\{0\},\quad l_2=\{1\}\times(0,1),\quad l_3=[0,1]\times\{1\},\quad l_4=\{0\}\times(0,1).$$
Then the boundary condition $\mathcal Bu(x,t)=b(x,t)$ is set as 
\begin{equation}
\begin{aligned}
 u&=b_1(x,t)\text{ on } l_1, &&u=b_3(x,t)\text{ on } l_3, \\
 \frac{\partial u}{\n}&=b_2(x,t) \text{ on } l_2, && \frac{\partial u}{\n}=b_4\text{ on } l_4,
 \end{aligned}
\end{equation}
where $b_4$ is a constant. Moreover, the source $f$, the initial condition $u_0$ and the potential constant $C_p$ are given; the space dependent drift coefficient $q(x)$ w.r.t $x_1$ direction is unknown.  

In the inverse problem of this work, the used measurement is the final time data of the solution, stated as  
\begin{equation}\label{data}
g(x):=u(x,T).
\end{equation}
Hence, the inverse problem investigated in this article could be described as follows: 
\begin{equation}\label{inverse_problem}
 \text{using data \eqref{data}}\ \text{to recover the drift}\ q(x)\ \text{in equation \eqref{PDE}}. 
\end{equation}

\subsection{Background and literature.}
Parabolic equations with drift terms arise naturally in mathematical models describing combined transport and diffusion processes. In statistical physics, such equations appear in Fokker-Planck models, where the drift represents deterministic transport induced by external forces and governs the evolution of probability densities \cite{risken1989}.
In fluid mechanics and mass transfer theory, convection-diffusion equations describe the interplay between molecular diffusion and advective transport
caused by fluid motion \cite{cussler2009}.
Drift-driven diffusion models also arise in population dynamics and chemotaxis,
where directed movement of organisms is influenced by environmental signals.
Furthermore, in financial mathematics, transformed parabolic equations
such as the Black-Scholes model contain drift terms representing deterministic
trends in asset price dynamics \cite{black1973}.
In these applications, the drift coefficient plays a crucial role in the system dynamics 
but is often not directly measurable. 
This naturally leads to inverse problems concerned with reconstructing the drift from available measurement data.
Inverse problems for convection-diffusion equations, including drift identification problems,
have been widely investigated under various observation settings, for example, \cite{isakov2006,cheng2009,li2013}.

From an analytical perspective, inverse drift reconstruction is challenging
due to the nonlinear dependence of the solution on the unknown coefficient and 
the instability of the problem with respect to data perturbations.
To overcome these difficulties, monotonicity-based operator methods provide a natural framework for inverse drift problems: one constructs a forward operator whose fixed points coincide with the unknown coefficient, and monotonicity derived from maximum principles yields both
uniqueness and a natural iterative reconstruction scheme. 
For parabolic inverse problems, early uniqueness results with final-time overdetermination can be traced back to classical works such as \cite{jones1962,frank1963,rundell1987,isakov1991}, while the monotonicity-based framework was surveyed in \cite{duchateau1995,tamburrino2006}.
A fixed-point monotonicity formulation was further developed for
time-dependent coefficients in fractional diffusion equations in
\cite{zhang2016,zhang2017,zhang2022}, where the unknown coefficient is characterized as a fixed point of a monotone operator, and the same structure leads to an iterative reconstruction scheme.
For inverse drift reconstruction, related terminal-observation results for one dimensional coefficients were obtained in \cite{CZZ2025}.
Other inverse problems for drift or convection terms in parabolic equations have also been studied under various observation settings; see, for example, \cite{deng2008,deng2011,bellassoued2021,sahoo2019}.
The present work extends this direction to a two-dimensional parabolic equation with mixed boundary conditions.

\subsection{Main results and outline.}

In this work we first construct an operator $K$ (see \eqref{operator}) whose one of fixed points is the unknown drift $q(x)$. Then we prove several properties for this operator, such as the well-definedness, equivalence and monotonicity. The uniqueness of inverse problem \eqref{inverse_problem} is established by showing the uniqueness of fixed points of $K$, and stated by Theorem \ref{theorem_uniqueness}. The proof of Theorem \ref{theorem_uniqueness} contains a specific iteration converging to the solution of inverse problem, which could be regarded as one advantage of such monotone operator method. From this it is natural to deduce the iterative algorithm to solve inverse problem \eqref{inverse_problem} numerically. The numerical results in Section \ref{section_num} demonstrate that the proposed method is effective and yields satisfactory reconstructions.   

The remainder of this article is structured as follows. In Section \ref{section_pre}, some preliminary results are collected and proved. For instance, the well-posedness of the forward problem of equation \eqref{PDE}, the maximum principles of elliptic and parabolic equations and the positivity results of equation \eqref{PDE}. We prove the uniqueness of inverse problem \eqref{inverse_problem} in Section \ref{section_unique}, using the operator $K$ defined in \eqref{operator}. Finally, several numerical experiments are taken in Section \ref{section_num} and the related numerical results are presented.

\section{Preliminaries.}\label{section_pre}
\subsection{Well-posedness of the forward problem.}


We now establish the well-posedness of equation \eqref{PDE}, including existence, uniqueness, and regularity of the solution. Our analysis follows the well-posedness framework of Theorem 2.1 in \cite{TDQ2016}, which is based on the classical parabolic theory developed in \cite{GS1997}.
In contrast to \cite{TDQ2016}, where a mixed boundary value problem with
time-dependent boundaries is considered under the divergence-free condition
$\nabla \cdot (q(x),0)=0$, we restrict ourselves to a fixed spatial domain with
time-independent boundary data and assume only $q \in L^\infty(\Omega)$.
Under these assumptions, the well-posedness result remains valid.
For completeness, we sketch the standard lifting and energy argument adapted to the present setting. For a boundary segment $\Gamma\subset\partial\Omega$, we write
\[
\Gamma_T:=\Gamma\times(0,T),\qquad
H^{r,s}(\Gamma_T):=L^2(0,T;H^r(\Gamma))\cap H^s(0,T;L^2(\Gamma)),
\]
with $r,s \in \mathbb{R}$.

\begin{lemma}\label{lemma:basic-wellposedness}
  Let $Q_T = \Omega\times(0,T]$. Assume that the source term satisfies $f \in L^2(\Omega)$, the initial data $u_0 \in H^1(\Omega)$, and the drift coefficient $q \in L^\infty(\Omega)$. The boundary data satisfy:
  \begin{itemize}
    \item $b_1, b_3 \in H^{3/2,\,3/4}(\Sigma_D)$, where $\Sigma_D = (l_1 \cup l_3) \times (0,T)$;
    \item $b_2, b_4 \in H^{1/2,\,1/4}(\Sigma_N)$, where $\Sigma_N = (l_2 \cup l_4) \times (0,T)$,
  \end{itemize}
  together with the compatibility conditions  
  $u_0(x) = b_1(x,0)$ on $l_1$ and $u_0(x) = b_3(x,0)$ on $l_3$.
  Then the forward problem of equation \eqref{PDE} admits a unique solution $u$ such that:
  $$\partial_t u \in L^2(Q_T), 
    \quad 
    u \in C([0,T]; H^1(\Omega)).$$
\end{lemma}

\begin{proof}
    We establish the well-posedness via a lifting argument and energy estimates.
    To handle the non-homogeneous Dirichlet boundary conditions, we introduce a
    lifting function $U_D$ that satisfies the prescribed boundary data and
    define $w = u - U_D$. More precisely, let $U_D$ be a function such that
    \begin{itemize}
      \item $U_D = b_1$ on $l_1$ and $U_D = b_3$ on $l_3$,
      \item $U_D = 0$ on $l_2 \cup l_4$,
      \item $U_D \in L^2(0,T; H^2(\Omega)) \cap H^1(0,T; L^2(\Omega))$.
    \end{itemize}
    The existence of such a lifting function follows from standard trace and extension results for parabolic problems;
    see, for instance, \cite{LSU1968}.
    
    Then $w$ satisfies homogeneous Dirichlet boundary
    conditions on $l_1 \cup l_3$ and
    $$
    \partial_t w - \Delta w + q(x)\partial_{x_1} w + C_p w
    = F \quad \text{in } \Omega \times (0,T],
    $$
    where
    $$
    F = f - \partial_t U_D + \Delta U_D - q(x)\partial_{x_1} U_D - C_p U_D.
    $$
    On $l_2 \cup l_4$, $w$ satisfies the Neumann condition
    $$
    \frac{\partial w}{\n}
    = b_N - \frac{\partial U_D}{\n}, \quad b_N=b_2 \,\text{on}\, l_2,\, b_N=b_4 \,\text{on} \,l_4.
    $$
    
    Multiplying the equation for $w$ by $w$ and integrating over $\Omega$ yields
    $$
    \int_\Omega w \partial_t w \, dx
    - \int_\Omega w \Delta w \, dx
    + \int_\Omega q(x) w \partial_{x_1} w \, dx
    + C_p \int_\Omega w^2 \, dx
    = \int_\Omega F w \, dx.
    $$
    
    Using integration by parts, the boundary condition, we obtain
    $$
    \frac{1}{2}\frac{d}{dt}\|w\|_{L^2(\Omega)}^2
    + \|\nabla w\|_{L^2(\Omega)}^2
    + C_p \|w\|_{L^2(\Omega)}^2
    = \int_\Omega F w \, dx
    - \int_\Omega q(x) w \partial_{x_1} w \, dx
    + \int_{l_2 \cup l_4}
    \Big(b_N - \frac{\partial U_D}{\n}\Big) w \, dS.
    $$
    
    The source term is estimated using Young's inequality:
    $$
    \left|\int_\Omega F w \, dx\right|
    \le \frac{1}{2}\|F\|_{L^2(\Omega)}^2
    + \frac{1}{2}\|w\|_{L^2(\Omega)}^2.
    $$
    
    Since $q \in L^\infty(\Omega)$, the drift term satisfies
    $$
    \left|\int_\Omega q(x) w \partial_{x_1} w \, dx\right|
    \le \frac{1}{2}\|\nabla w\|_{L^2(\Omega)}^2
    + \frac{\|q\|_{L^\infty(\Omega)}^2}{2}\|w\|_{L^2(\Omega)}^2.
    $$
    
    The boundary term on $l_2 \cup l_4$ is estimated by the trace theorem and Young's inequality: For any $\varepsilon>0$
    $$
    \left|\int_{l_2 \cup l_4}
    \Big(b_N - \frac{\partial U_D}{\n}\Big) w \, dS\right|
    \le \varepsilon \|\nabla w\|_{L^2(\Omega)}^2
    + C_\varepsilon\Big(\|b_N\|_{H^{1/2,1/4}(\Sigma_N)}^2
    + \|U_D\|_{L^2(0,T;H^2(\Omega))}^2\Big),
    $$
    
    Combining the above estimates yields
    $$
    \frac{d}{dt}\|w\|_{L^2(\Omega)}^2
    + \|\nabla w\|_{L^2(\Omega)}^2
    \le C_1 \|w\|_{L^2(\Omega)}^2 + \|F\|_{L^2(\Omega)}^2,
    $$
    where $C_1>0$ depends on $\|q\|_{L^\infty(\Omega)}$ and $C_p$.
    By Gronwall's inequality,
    $$
    \|w(t)\|_{L^2(\Omega)}^2
    \le e^{C_1 t}
    \left(
    \|w(0)\|_{L^2(\Omega)}^2
    + \int_0^t \|F(s)\|_{L^2(\Omega)}^2 \, ds
    \right),
    $$
    with $w(0) = u_0 - U_D(0)$.
    
    Higher regularity, in particular $\partial_t w \in L^2(Q_T)$, follows by
    testing the equation with $\partial_t w$ and applying standard energy
    arguments for parabolic equations. Consequently,
    $w \in C([0,T];H^1(\Omega))$ and $\partial_t w \in L^2(Q_T)$.
    Since $u = w + U_D$, the asserted regularity for $u$ follows.
    Uniqueness is a direct consequence of linearity and the energy estimate. The proof is complete.
\end{proof}

In addition to the fundamental well-posedness result in Lemma \ref{lemma:basic-wellposedness}, we next investigate the continuity of spatial derivatives of the solution, which plays a crucial role in the subsequent analysis of the inverse problem.
Differentiating equation \eqref{PDE} with respect to $x_1$ leads to the equation for $w = \partial_{x_1}u$ as
  \begin{equation}\label{PDE_x1}
    \begin{cases}
     \begin{aligned}
      (\partial_t-\Delta+q(x)\partial_{x_1}+(C_p+\partial_{x_1}q))w(x,t)&=\partial_{x_1}f, &&(x,t)\in \Omega\times(0,T];\\
      w(x,t)&=\partial_{x_1}b_1,&&(x,t)\in l_1\times(0,T];\\
      w(x,t)&=b_2,&&(x,t)\in l_2\times(0,T];\\
      w(x,t)&=\partial_{x_1}b_3,&&(x,t)\in l_3\times(0,T];\\
      w(x,t)&=-b_4,&&(x,t)\in l_4\times(0,T];\\
      w(x,0)&=\partial_{x_1}u_0(x),&&x\in\Omega.
     \end{aligned}
    \end{cases}
   \end{equation}

The regularity of $w$ is established using the boundary De Giorgi theory for parabolic equations with Dirichlet boundary conditions.
Although the original problem \eqref{PDE} is subject to mixed boundary
conditions, the differentiated problem \eqref{PDE_x1} is equipped with
Dirichlet boundary data for the new unknown $w$ on the whole lateral
boundary.
In particular, we use the qualitative consequence of the boundary
continuity theorem, namely that functions in a boundary De Giorgi class
with continuous Dirichlet data are continuous up to the lateral boundary.
A representative statement of this type is given by
\cite[Theorem~7.1, Chapter~10]{DEG2023}; see also the corresponding
parabolic boundary continuity results in
\cite{L1996,lieberman1993,dibenedetto2012}.
Under the assumptions stated below,
Since $\Omega=(0,1)^2$ automatically satisfies the required geometric
density condition, and the differentiated equations below are linear
parabolic equations with bounded coefficients and continuous Dirichlet
boundary data, their weak solutions belong to the corresponding
parabolic boundary De Giorgi classes. Therefore the boundary continuity
theory applies.

\begin{lemma}\label{lemma:continuity_ux1}
    Let $w = \partial_{x_1}u$ satisfy equation (\ref{PDE_x1}). Assume that $q\in W^{1,\infty}(\Omega)$, $f\in W^{1,\infty}(\Omega)$ with $\partial_{x_1}f\ge 0$,
    and $u_0 \in H^1(\Omega)$ with $\partial_{x_1}u_0\ge 0$. 
    Assume further that the Dirichlet boundary data in (\ref{PDE_x1}) are continuous on the corresponding portions of $\partial\Omega$.
    Then $w$ belongs to the boundary De Giorgi class and consequently
    $$
    w\in C(\overline{\Omega}\times(0,T]),
    $$
    and in particular $w(\cdot,T)\in C(\overline{\Omega})$.
\end{lemma}

Having established the continuity of the spatial derivative $\partial_{x_1}u$ in Lemma \ref{lemma:continuity_ux1}, we next consider the mixed derivative $\partial^2_{x_1t}u$, which is crucial for the convergence analysis of 
the fixed point iteration. 
The model of $w=\partial^2_{x_1t}u$ is written as 
 \begin{equation}\label{PDE_x1t}
 \begin{cases}
  \begin{aligned}
   (\partial_t-\Delta+q(x)\partial_{x_1}+((C_p+\partial_{x_1}q)))w(x,t)&=0, &&(x,t)\in \Omega\times(0,T];\\
   w(x,t)&=\partial^2_{x_1t}b_1,&&(x,t)\in l_1\times(0,T];\\
   w(x,t)&=\partial_tb_2,&&(x,t)\in l_2\times(0,T];\\
   w(x,t)&=\partial^2_{x_1t}b_3,&&(x,t)\in l_3\times(0,T];\\
   w(x,t)&=0,&&(x,t)\in l_4\times(0,T],  
  \end{aligned}
 \end{cases}
\end{equation}
with 
$$ w(x,0)=\partial_{x_1}f+\Delta(\partial_{x_1}u_0)-q(x)\partial^2_{x_1x_1}u_0-(C_p+\partial_{x_1}q)\partial_{x_1}u_0,\quad x\in\Omega.$$ 

As in the previous case, the regularity of $w$ follows from the boundary De Giorgi class theory for parabolic equations with Dirichlet data.
The corresponding continuity result is stated below.

\begin{lemma}\label{lemma:continuity_w_xt}
  Let $w=\partial^2_{x_1t}u$ satisfy equation (\ref{PDE_x1t}) 
  Assume that $q\in W^{1,\infty}(\Omega)$ and $C_p\ge 0$.
  Assume further that
  $$
  \partial^2_{x_1 t} b_1\in C(\overline{l_1} \times (0,T]), \quad
  \partial_t b_2\in C(\overline{l_2} \times (0,T]), \quad
  \partial^2_{x_1 t} b_3\in C(\overline{l_3} \times (0,T]).
  $$
  Then $w$ belongs to the boundary De Giorgi class associated with the Dirichlet data in \eqref{PDE_x1t}. Consequently,
  $$
  w\in C(\overline{\Omega}\times(0,T]).
  $$
  In particular, there exists a constant $M>0$, depending only on
  $\Omega$, $T$, $\|q\|_{W^{1,\infty}(\Omega)}$,
  $C_p$ and the boundary data, such that
  $$
  \|w(\cdot,T)\|_{L^\infty(\Omega)} \le M.
  $$
\end{lemma}

From Lemmas~\ref{lemma:basic-wellposedness}--\ref{lemma:continuity_w_xt}, we could summarize some required assumptions as follows.  

\begin{assumption}\label{assumption_regularity}
  Throughout this paper, we impose the regularity and compatibility assumptions
  on the data required to ensure the well-posedness and regularity results
  established in Lemmas~\ref{lemma:basic-wellposedness}--\ref{lemma:continuity_w_xt}.
  In particular, the drift coefficient satisfies $q \in W^{1,\infty}(\Omega)$,
  the initial datum satisfies $u_0 \in H^1(\Omega)$,
  the source term and boundary data possess sufficient regularity
  and compatibility at $t=0$.
  These assumptions guarantee that all derivatives appearing in the
  differentiated problems are well defined.
\end{assumption}

\subsection{Positivity results.}

For the subsequent analysis, we recall the strong maximum principles for elliptic and parabolic equations, together with Hopf's lemma for elliptic equations, from Evans\cite[Chapter 6]{E2010}.
\begin{lemma}\label{lemma_elliptic_max}
  [Strong maximum principle for elliptic equations]
  Let $U \subset \mathbb{R}^n$ be a connected, open and bounded domain. Suppose $u \in C^2(U) \cap C(\overline{U})$ satisfies
  $$
  Lu := - \sum_{i,j=1}^n a_{ij}(x) \partial_{ij} u + \sum_{i=1}^n b_i(x) \partial_i u + c(x) u
  $$
  where $(a_{ij})$ is uniformly elliptic and $c \ge 0$. 
  \begin{enumerate}[(i)]
    \item If $Lu\le 0$ in $U$ and $u$ attains a nonnegative maximum over $\overline{U}$ at an interior point of $U$, then $u$ is constant in $U$.
    
    \item If $Lu\ge 0$ in $U$ and $u$ attains a nonpositive minimum over $\overline{U}$ at an interior point of $U$, then $u$ is constant in $U$.
  \end{enumerate}
\end{lemma}

\begin{lemma}\label{lemma_parabolic_max}
  [Strong maximum principles for parabolic equations]
  Let $U \subset \mathbb{R}^n$ be a connected domain and $T>0$. 
  Define $U_T := U \times (0,T]$ and $\overline{U}_T := \overline{U} \times [0,T]$. 
  Suppose $u \in C^{2,1}(U_T) \cap C(\overline{U}_T)$ and $L$ is a uniformly elliptic operator of the same form as in Lemma\ref{lemma_elliptic_max} with $c\ge 0$ in $U_T$.  
  \begin{enumerate}[(i)]
    \item If $u_t+Lu\le 0$ in $U_T$ and $u$ attains a nonnegative maximum over $\overline{U}_T$ at a point $(x_0,t_0)\in U_T$, then $u$ is constant in $U\times[0,t_0]$.
    
    \item If $u_t+Lu\ge 0$ in $U_T$ and $u$ attains a nonpositive minimum over $\overline{U}_T$ at a point $(x_0,t_0)\in U_T$, then $u$ is constant in $U\times[0,t_0]$.
  \end{enumerate}
\end{lemma}

\begin{lemma}\label{lemma_hopf}
    [Hopf's lemma]
    Let $U\subset\mathbb{R}^n$ be a bounded domain and let $L$ is a uniformly elliptic operator of the same form as in Lemma\ref{lemma_elliptic_max} with $c\ge 0$ in $U$.
    Assume that $u\in C^2(U)\cap C^1(\overline U)$ satisfies $Lu\le0$ in $U$.
    If $u$ attains a strict positive maximum at some boundary point
    $x^0\in\partial U$ where the interior ball condition holds, then
    \[
    \frac{\partial u(x^0)}{\n}>0.
    \] 
\end{lemma}

Before proving the positivity results of solution $u$, we state several assumptions as follows.

 \begin{assumption}\label{assumption_positivity}
For equation \eqref{PDE} we assume that Assumption \ref{assumption_regularity} be valid. Moreover, the source $f$, initial condition $u_0$, boundary conditions $b_1, b_2, b_3, b_4$, the exact drift $q$ and the potential constant $C_p$ satisfy the following assumptions.  
  \begin{itemize}
   \item [$(a)$] $q\in C^1 (\Omega)$ and $\|q\|_{C^1(\Omega)}\le M$, where $M$ is a sufficiently large constant. 
  \item [$(b)$] $C_p>M$, where $M$ is given by $(a)$.
  \item [$(c)$] $\partial_{x_1}b_j\ge 0,\ \partial^2_{x_1t}b_j\ge 0,\ \partial_t b_j\ge 0$ on $l_j\times(0,T],\ j=1,3$.  
   \item [$(d)$] $b_2\ge 0,\ \partial_tb_2\ge 0,\ \partial_tb_2>0$ on $l_2\times(0,T]$, and the constant $b_4$ is strictly negative.    
  \item [$(e)$] $u_0\in C^3(\Omega)$, $\partial_{x_1}u_0(x)\ge 0$ on $\Omega$ and $C_0:=\max_{0\le j\le 3}\{\|u_0\|_{C^j(\Omega)}\}<\infty$.
  \item [$(f)$] $f\ge (1+M+C_p)C_0$ and $\partial_{x_1}f\ge (1+2M+C_p)C_0$ on $\Omega$. 
 \end{itemize}
\end{assumption}

\begin{lemma}\label{lemma_positive}
Under Assumptions \ref{assumption_positivity} and \ref{assumption_regularity}, we have the next positivity results for equation \eqref{PDE}.
\begin{itemize}
 \item [$(a)$] $\partial_{x_1} u\ge 0$ on $\Omega\times(0,T]$.
 \item [$(b)$]$\exists C_m>0$ such that $\partial_{x_1}u(x,T)>C_m$ on $\Omega$. 
  \item [$(c)$]$\partial_tu(x,t)\ge 0$ on $\Omega\times(0,T]$.
 \item [$(d)$]$\partial^2_{x_1t}u(x,t)\ge 0$ on $\Omega\times(0,T]$.
\end{itemize}
\end{lemma}
\begin{proof}
 For $(a)$, with Assumption \ref{assumption_positivity} the model of $w=\partial_{x_1} u$ is given as \eqref{PDE_x1}.
Assumption \ref{assumption_positivity} gives the nonnegativity of $(C_p+\partial_{x_1}q)$, which together with Lemma \ref{lemma_parabolic_max} generates the statement $(a)$. 

For $(b)$, the boundary condition of equation \eqref{PDE_x1} and Lemma \ref{lemma_parabolic_max} yield that $w(x,T)>0$ on $\overline\Omega$. Otherwise, if $w(x^*,T)=0$ with $x^*\in \Omega$, it would be the  interior nonpositive minimum. Then by Lemma \ref{lemma_parabolic_max} we have that $w\equiv 0$, which contradicts with the strict positive boundary condition. The result $w(x,T)>0$ on $\overline\Omega$ and the continuity of $w$ ensured by Lemma \ref{lemma:continuity_ux1} give the conclusion $(b)$. 

For $(c)$, we could give the model of $w=\partial_t u$ as 
\begin{equation}\label{PDE_t}
 \begin{cases}
  \begin{aligned}
   (\partial_t-\Delta+q(x)\partial_{x_1}+C_p)w(x,t)&=0, &&(x,t)\in \Omega\times(0,T];\\
   w(x,t)&=\partial_t b_1\ge 0,&&(x,t)\in l_1\times(0,T];\\
   \frac{\partial w}{\n}(x,t)&=\partial_tb_2>0,&&(x,t)\in l_2\times(0,T];\\
   w(x,t)&=\partial_t b_3\ge 0,&&(x,t)\in l_3\times(0,T];\\
   \frac{\partial w}{\n}(x,t)&=0,&&(x,t)\in l_4\times(0,T];\\
   w(x,0)&=f+(\Delta-q(x)\partial_{x_1}-C_p)u_0(x)\ge 0,&&x\in\Omega.
  \end{aligned}
 \end{cases}
\end{equation}
Assume that  
$$w(x^*,t^*):=\min\{w(x,t):(x,t)\in\overline\Omega\times[0,T]\}<0.$$ 
Lemma \ref{lemma_parabolic_max} implies that $(x^*,t^*)\in \partial\Omega\times[0,T]\cup \overline\Omega\times\{0\}$; while the initial and boundary conditions give that $x^*\in l_2\cup l_4$ and $t^*>0$. For the case of $x^*\in l_2$, with the result $\partial_tb_2>0$, we could find $x^\dag\in\Omega$ such that $x^\dag$ has the same $x_2-$coordinate with $x^*$ and $w(x^\dag,t^*)<w(x^*,t^*)$, which contradicts with the minimal assumption. For the case that $x^*\in l_4$, the minimal assumption and the boundary condition $\frac{\partial w}{\n}(x,t)=0$ on $l_4$ ensure that $(\partial_t-\Delta+q(x)\partial_{x_1})w\le 0$ on $(x^*,t^*)$; while we could have $C_pw(x^*,t^*)<0$ from the strict positivity of $C_p$. These gives that 
$$(\partial_t-\Delta+q(x)\partial_{x_1}+C_p)w(x^*,t^*)<0,$$ 
which contradicts with equation \eqref{PDE_t}. Hence, we could deduce the nonnegativity of $w=\partial_t u$.

For $(d)$, the model of $w=\partial^2_{x_1t}u$ is written as \ref{PDE_x1t} 
With the result $(C_p+\partial_{x_1}q)\ge 0$, Lemma \ref{lemma_parabolic_max} gives the nonnegativity of $w=\partial^2_{x_1t}u$. The proof is complete. 
\end{proof}

\section{Analysis of inverse problem \eqref{inverse_problem}.}\label{section_unique}
Next we would display how to solve inverse problem \eqref{inverse_problem} by the monotone operator method.

\subsection{Operator $K$.}
From equation \eqref{PDE} and data \eqref{data}, we define the operator $K$ as 
\begin{equation}\label{operator}
 K\psi=\frac{f(x)-\partial_tu(x,T;\psi)+\Delta g(x)-C_pg(x)}{\partial_{x_1}g(x)},
\end{equation}
with domain
\begin{equation}\label{domain}
\mathcal D_K:=\Big\{\psi\in C^1(\Omega):\psi\le \frac{f(x)+\Delta g(x)-C_pg(x)}{\partial_{x_1}g(x)}\Big\}.
\end{equation}
Here the notation $u(x,T;\psi)$ means the solution $u$ of equation \eqref{PDE} with drift term $\psi$. The well-definedness of $K$ is ensured by Lemma \ref{lemma_positive}, from which the denominator $\partial_{x_1}g(x)$ is strictly positive.

From the next lemma, we could solve inverse problem \eqref{inverse_problem} by considering the fixed point problem of $K$. 
\begin{lemma}\label{lemma_equivalence}
With \eqref{operator} and \eqref{domain}, for $q\in \mathcal D_K$, $u(x,T;q)=g(x)$ is valid if and only if $Kq=q$. 
\end{lemma}

\begin{proof}
 If $u(x,T;q)=g(x)$, it is obvious that $q$ is a fixed point of $K$ from \eqref{PDE} and \eqref{operator}. 
 
 Given a fixed point $q\in \mathcal D_K$ of $K$, now we need to show it satisfies $u(x,T;q)=g(x)$. 
 Setting $w(x)=u(x,T;q)-g(x)$, we have that 
 \begin{equation*}
 \begin{cases}
  \begin{aligned}
   (-\Delta +q(x)\partial_{x_1}+C_p)w(x)&=0, &&x\in \Omega;\\
   \mathcal Bw(x)&=0,&&x\in\partial\Omega.
  \end{aligned}
 \end{cases}
\end{equation*}
Assume that $w\not\equiv 0$. By Lemma~\ref{lemma_elliptic_max},
$w$ cannot attain a positive maximum in the interior of $\Omega$.
Since $w=0$ on $l_1\cup l_3$, any positive maximum must be attained at some point
of $l_2\cup l_4$. Then Lemma~\ref{lemma_hopf} implies
$\frac{\partial w}{\n}>0$ at that point, contradicting
$\frac{\partial w}{\n}=0$ on $l_2\cup l_4$.
Hence $w\le 0$ in $\Omega$.
Applying the same argument to $-w$, we obtain $w\ge 0$ in $\Omega$.
Therefore $w\equiv 0$ in $\Omega$ and hence $u(x,T;q) = g(x)$. This completes the proof.
\end{proof}

\subsection{Monotonicity.}
The monotonicity of operator $K$ means that $K$ could keep the monotone relation on $\mathcal D_K$, which is stated by the next lemma.  
\begin{lemma}\label{lemma_monotone}
 With operator $K$ in \eqref{operator} and domain $\mathcal D_K$ in \eqref{domain}, for $\psi_1, \psi_2\in \mathcal D_K$, the result $\psi_1\le \psi_2$ leads to $K\psi_1\le K\psi_2$ on $\Omega$.
\end{lemma}
\begin{proof}
 Setting $w=\partial_tu(x,t;\psi_1)-\partial_tu(x,t;\psi_2)$, it satisfies 
 \begin{equation*}
 \begin{cases}
  \begin{aligned}
   (\partial_t-\Delta +\psi_1(x)\partial_{x_1}+C_p)w(x,t)&=(\psi_2-\psi_1)\partial^2_{x_1t}u(x,t;\psi_2), &&(x,t)\in \Omega\times(0,T];\\
   \mathcal Bw(x,t)&=0,&&(x,t)\in\partial\Omega\times(0,T];\\
   w(x,0)&=(\psi_2-\psi_1)\partial_{x_1}u_0(x),&&x\in\Omega.
  \end{aligned}
 \end{cases}
\end{equation*}
From the definition of operator $K$, it is sufficient to show the nonnegativity of $w$. 
Assume that 
$$w(x^*,t^*):=\min\{w(x,t):(x,t)\in\overline\Omega\times[0,T]\}.$$
From Lemma \ref{lemma_parabolic_max} and Assumption \ref{assumption_positivity}, we have that 
$$(x^*,t^*)\in \{0,1\}\times(0,1)\times(0,T].$$

For the case of $x^*\in \{0\}\times(0,1)$, the minimum result and the strict negativity yield that 
$$(\partial_t-\partial^2_{x_2x_2} u+C_p)w(x^*,t^*)<0;$$
while the boundary condition on $\{0\}\times(0,1)$ together with the minimum  gives that 
$$(-\partial^2_{x_1x_1}+\psi_1(x)\partial_{x_1})w(x^*,t^*)\le 0.$$
The above two inequalities lead to the result that $(\psi_2-\psi_1)\partial^2_{x_1t}u(x,t;\psi_2)<0$, which contradicts with Lemma \ref{lemma_positive} and the fact $\psi_1\le \psi_2$. 
The proofs for the case of $x^*\in \{1\}\times(0,1)$ could be deduced analogously. Hence we have that $w\ge 0$ on $\Omega\times(0,T]$, which leads to the desired result. The proof is complete.
\end{proof}

\subsection{The specific iteration and uniqueness theorem.}

First we build a uniqueness result under monotone relation. 
\begin{lemma}\label{lemma_uniqueness_1}
 Recalling $K$ and $\mathcal D_K$ in \eqref{operator} and \eqref{domain}, given $q_1, q_2\in \mathcal D_K$ with $q_1\le q_2$, if $q_1$ and $q_2$ are both fixed points of $K$, then $q_1=q_2$ on $\Omega$.  
\end{lemma}

\begin{proof}
  The model for $w=u(x,t;\psi_1)-u(x,t;\psi_2)$ could be given as 
 \begin{equation*}
 \begin{cases}
  \begin{aligned}
   (\partial_t-\Delta +\psi_1(x)\partial_{x_1}+C_p)w(x,t)&=(\psi_2-\psi_1)\partial_{x_1}u(x,t;\psi_2), &&(x,t)\in \Omega\times(0,T];\\
   \mathcal Bw(x,t)&=0,&&(x,t)\in\partial\Omega\times(0,T];\\
   w(x,0)&=w(x,T)=0,&&x\in\Omega.
  \end{aligned}
 \end{cases}
\end{equation*}
The equality $w(x,T)=0$ follows from Lemma \ref{lemma_equivalence}. The proof of Lemma \ref{lemma_monotone} gives the nonnegativity of $\partial_t w$, which together with $ w(x,0)=w(x,T)=0$ yields that $w\equiv 0$ on $\Omega\times(0,T]$. Then we obtain that $(\psi_2-\psi_1)\partial_{x_1}u(x,t;\psi_2)\equiv 0$, which leads to $\psi_1=\psi_2$ since Lemma \ref{lemma_positive} ensures the strict positivity of $\partial_{x_1}u$. The proof is complete.
\end{proof}

Next we define the iteration from upper bound of $\mathcal D_K$ as follows: 
\begin{equation}\label{iteration}
 q_0=\frac{f(x)+\Delta g(x)-C_pg(x)}{\partial_{x_1}g(x)},\quad q_{n+1}=Kq_n,\quad n\in \mathbb N.
\end{equation}

Now we could state the uniqueness theorem of inverse problem \eqref{inverse_problem}. 

\begin{theorem}\label{theorem_uniqueness}
 With \eqref{operator} and \eqref{domain}, fixing a fixed point $q\in\mathcal D_K$ of $K$, the sequence $\{q_n\}_{n=1}^\infty$ generated by \eqref{iteration} would converge to $q$ decreasingly. Therefore, there is at most one fixed point of $K$ in $\mathcal D_K$. By Lemma \ref{lemma_equivalence}, we could obtain the uniqueness of inverse problem \eqref{inverse_problem}. 
\end{theorem}

\begin{proof}
 If $q\in \mathcal D_K$ is the solution of inverse problem \eqref{inverse_problem}, it should be one fixed point of $K$ by Lemma \ref{lemma_equivalence}. From \eqref{iteration}, with the fact that the initial guess $q_0\in\mathcal D_K$ is the upper bound of the domain $\mathcal D_K$, we have $q\le q_0$, which leads to $q=Kq\le Kq_0=q_1$ by Lemma \ref{lemma_monotone}. Also, the nonnegativity of $\partial_t u(x,T;q_0)$ ensured by Lemma \ref{lemma_positive} leads to $q_1\le q_0$, and Lemmas~\ref{lemma:continuity_ux1}-\ref{lemma:continuity_w_xt} guarantee the required smoothness of $q_1$. To sum up, we have $q_1\in\mathcal D_K$ and $q\le q_1\le q_0$. Using Lemma \ref{lemma_monotone} again, it holds that 
 \begin{equation*}
  q=Kq\le Kq_1=q_2\le Kq_0=q_1\le q_0.
 \end{equation*}
Continuing this argument, it gives that $\{q_n\}_{n=0}^\infty$ is a decreasing sequence and has a lower bound $q$, which yields the pointwise convergence of $\{q_n\}_{n=0}^\infty$. The limit is denoted by $\tilde q$ and obviously $q\le \tilde q\le q_0$. 

Now the task is to show that $q=\tilde q$. From the monotone convergence theorem, we have that $\|q_n-\tilde q\|_{L^2(0,1)}\to 0$. With triangle inequality, we have 
\begin{align*}
 \|K\tilde q-\tilde q\|_{L^2(0,1)}&\le \|K\tilde q-Kq_n\|_{L^2(0,1)}+\|Kq_n-\tilde q\|_{L^2(0,1)}\\
 &=\|K\tilde q-Kq_n\|_{L^2(0,1)} +\|q_{n+1}-\tilde q\|_{L^2(0,1)}\\
 &=:I_1+I_2.
\end{align*}
It is obvious that $I_2\to 0$ as $n\to \infty$. For $I_1$, with Lemma \ref{lemma_positive} we have 
\begin{align*}
 I_1=\|[\partial_t u(x,T;\tilde q)-\partial_t u(x,T;q_n)]/\partial_{x_1}g(x)\|_{L^2(0,1)}
 \le C \|\partial_t u(x,T;\tilde q)-\partial_t u(x,T;q_n)\|_{L^2(0,1)},
\end{align*}
where the constant $C$ depends on $m$. Setting $w=\partial_t u(x,t;\tilde q)-\partial_t u(x,t;q_n)$, the proof of Lemma \ref{lemma_monotone} gives the model of $w$ as 
\begin{equation*}
 \begin{cases}
  \begin{aligned}
   (\partial_t-\Delta +\tilde q(x)\partial_{x_1}+C_p)w(x,t)&=(q_n-\tilde q)\partial^2_{x_1t}u(x,t;\psi_2), &&(x,t)\in \Omega\times(0,T];\\
   \mathcal Bw(x,t)&=0,&&(x,t)\in\partial\Omega\times(0,T];\\
   w(x,0)&=(q_n-\tilde q)\partial_{x_1}u_0(x),&&x\in\Omega.
  \end{aligned}
 \end{cases}
\end{equation*}

Testing the equation for $w$ with $w$ and integrating over $\Omega$,
an application of Young's inequality and Gr\"onwall's lemma yields the energy estimate
\begin{equation}\label{equ:w_L2_basic}
\|w(\cdot,T)\|_{L^2(\Omega)}
\le
C\Big(
\|w(\cdot,0)\|_{L^2(\Omega)}
+
\|(q_n-\tilde q)\,\partial^2_{x_1t}u(\cdot;\tilde q)\|_{L^2(Q_T)}
\Big),
\end{equation}
where $C>0$ is a constant.

By the well-posedness and regularity result stated in
Lemma~\ref{lemma:basic-wellposedness}, estimate \eqref{equ:w_L2_basic} yields
\begin{align}\label{equ:w_L2T_step1}
\|w(\cdot,T)\|_{L^2(\Omega)}
&\le
C\Big(
\|w(\cdot,0)\|_{L^2(\Omega)}
+
\|(q_n-\tilde q)\,\partial^2_{x_1t}u(\cdot;\tilde q)\|_{L^2(Q_T)}
\Big) \nonumber \\
&\le
C\Big(
\|\partial_{x_1}u_0\|_{L^\infty(\Omega)}
+
T^{1/2}\|\partial^2_{x_1t}u\|_{L^\infty(Q_T)}
\Big)
\|q_n-\tilde q\|_{L^2(\Omega)}.
\end{align}

By Lemma~\ref{lemma:continuity_w_xt} and Assumption~\ref{assumption_positivity},
both $\partial_{x_1}u_0$ and the mixed derivative
$\partial^2_{x_1t}u$ are bounded.
Consequently, estimate \eqref{equ:w_L2T_step1} simplifies to
\begin{equation}\label{equ:w_L2T_final}
\|w(\cdot,T)\|_{L^2(\Omega)}
\le
C\,\|q_n-\tilde q\|_{L^2(\Omega)},
\end{equation}
where the constant $C>0$ depends only on the a priori data of the problem.

So $I_1\le C\|q_n-\tilde q\|_{L^2(\Omega)}$, which converges to zero as $n\to \infty$. To sum up, we could deduce that $\|K\tilde q-\tilde q\|_{L^2(0,1)}=0$, which gives that $\tilde q$ is a fixed point of $K$. Note that the pointwise convergence only ensures that $\tilde q\in L^\infty(0,1)$. 
We further rely on the interior and boundary regularity results established in
Lemmas~\ref{lemma:continuity_ux1} and~\ref{lemma:continuity_w_xt},
which imply $$\partial_t u(\cdot,T;\psi)\in C^1(\Omega).$$ Consequently, $K\tilde q\in C^1(\Omega)$.
With $\|K\tilde q-\tilde q\|_{L^2(0,1)}=0$, we could pick the smooth version of $\tilde q$ to make sure $\tilde q\in C^1$. 
Now we have proved that $q$ and $\tilde q$ are both fixed points of $K$ in $\mathcal D_K$ and they satisfy $q\le \tilde q$. Applying Lemma \ref{lemma_uniqueness_1}, we have $q=\tilde q$. 

The uniqueness could be deduced straightforwardly. For two solutions $q$ and $\hat q$ in $\mathcal D_K$ of inverse problem \eqref{inverse_problem}, Lemma \ref{lemma_equivalence} implies that they are both the fixed points of $K$. With the above proofs, we obtain that $q_n\to q$ and $q_n\to \hat q$ simultaneously. This gives $q=\hat q$ and completes the proof. 
\end{proof}

\section{Numerical experiments}\label{section_num}

In this section, We report numerical experiments for the proposed reconstruction method.
All computations use uniform Cartesian grids and finite differences for the
forward solver in Algorithm~\ref{alg:inverse_drift}.

Let $\Omega=(0,1)^2$ be discretized by a uniform grid with steps $h_x,h_y$, where $(x,y)$ denotes the spatial variable. Let $t^n=n\tau$, $n=0,1,\dots,N_t$, partition $[0,T]$.
For a given drift coefficient $q$, the forward problem
\begin{equation}\label{eq:forward_disc}
\partial_t u - \Delta u + q(x,y)\,\partial_{x}u + C_p u = f(x,y)
\end{equation}
is discretized in time by backward Euler,
\begin{equation}\label{eq:time_disc}
\frac{u^{n+1}-u^n}{\tau}
- \Delta u^{n+1}
+ q(x,y)\,\partial_{x}u^{n+1}
+ C_p u^{n+1}
= f(x,y),
\qquad n=0,1,\dots,N_t-1.
\end{equation}
In space, we use centered finite differences for $\Delta$ and $\partial_{x_1}u$
at interior grid points. Dirichlet conditions are imposed on $y=0$ and $y=1$
by directly enforcing the boundary values at the corresponding nodes.
Neumann conditions on $x=0$ and $x=1$ are enforced by replacing the corresponding boundary rows with second-order one-sided finite difference approximations.
Each time step solves
\begin{equation}\label{eq:linear_system}
\Big(
\frac{1}{\tau}I
- \Delta_h
+ Q\,D_{x,h}
+ C_p I
\Big) u^{n+1}
=
\frac{1}{\tau}u^n + f,
\end{equation}
where $\Delta_h$ and $D_{x,h}$ are the discrete Laplace and first-order
operators, and $Q=\mathrm{diag}(q)$. The solver is called each iteration to
compute the forward solution $u$ and a backward-difference approximation of $\partial_t u(\cdot,T)$.

\begin{algorithm}[htbp]
    \caption{An iterative algorithm for recovering the drift term $q$}
    \label{alg:inverse_drift}
    \begin{algorithmic}[1]
    \REQUIRE Source term $f$, terminal observation $g=u(\cdot,T)$, 
    initial data $u_0$, potential constant $C_p$.
    \ENSURE Approximate drift term $q_N$
    
    \STATE Set the initial guess
    $$
    q_0 = \frac{f + \Delta g - C_p g}{\partial_{x} g}.
    $$
    
    \FOR{$k = 1,2,\ldots,N$}
    
    \STATE 
    Compute the solution $u(\cdot,t;q_{k-1})$ by solving the forward problem (\ref{PDE}) with $q_{k-1}$;
    
    \STATE Update the drift term by
    $$
    q_k
    =
    K q_{k-1}
    =
    \frac{f - \partial_t u(\cdot,T;q_{k-1})
    + \Delta g - C_p g}
    {\partial_{x_1} g}.
    $$
    
    \STATE Check the stopping criterion
    $\|q_k - q_{k-1}\|_{L^2(\Omega)} \le \varepsilon_0$ for some $\varepsilon_0 > 0$.
    \ENDFOR
    
    \STATE \textbf{Output} the recovered drift term $q_N$.
    
    \end{algorithmic}
    \end{algorithm}    

The experiments are used to show three main points:
(i) the convergence of the proposed iteration for noise-free data,
(ii) its stability under noisy terminal measurements,
and (iii) the monotone improvement of the iterates in the first few steps. 
The detailed examples are given as follows.

\begin{enumerate}[(a)]
  \item The smooth coefficient is chosen as
  \[
  q_{1}(x,y) = 1 + \sin(\pi x)\sin(\pi y).
  \]

  \item The piecewise constant coefficient is given by
  \[
  q_{2}(x,y) =
  \begin{cases}
  1.4, & |x-0.6|\le 0.18,\ \ |y-0.4|\le 0.18,\\
  1,   & \text{otherwise}.
  \end{cases}
  \]

  \item A binary coefficient supported on a Chinese-character-shaped region, namely
  \[
    q_3(x,y)=1+0.4\,\chi_{D_{\mathrm{char}}}(x,y),
  \]
  where $D_{\mathrm{char}}\subset(0,1)^2$ denotes the region shown in the corresponding figure.

\end{enumerate}

The potential constant is fixed as $C_p=5$, the final time is $T=1$,
and the stopping tolerance is set to $10^{-13}$.
The initial and boundary data are chosen in a compatible way.
Let
$a(t)=e^{\beta t}, \beta=1.$
The Dirichlet boundary conditions are prescribed by
\[
b_1(x,t)=a(t)e^x,\qquad b_3(x,t)=a(t)e^x,
\]
on \(y=0\) and \(y=1\), respectively.
The Neumann boundary conditions are given by
\[
b_2(y,t)=a(t)e,\qquad b_4(y,t)=-a(t),
\]
on \(x=1\) and \(x=0\), respectively, where \(b_4\) is defined with respect to the outward normal derivative.
The initial value is taken as
\[
u_0(x,y)=a(0)e^x=e^x,
\]
and is further corrected at the boundary nodes so that  the discrete Dirichlet and Neumann compatibility conditions at \(t=0\) are satisfied.
The source term is given by
\[
f(x,y)= 5\sin(\pi x)\sin(\pi y).
\]
To avoid an inverse crime, the terminal data are generated on a finer
$100\times100$ grid, while the inverse reconstruction is solved on a coarser grid $60\times60$.
All figures are displayed on the reconstruction grid used in the corresponding experiment.

For noisy-data experiments, the terminal data are perturbed by
\[
g^\delta = g + \delta \|g\|_{L^\infty(\Omega)}\,\xi,
\]
where $\xi$ is a random function with values in $[-\tfrac12,\tfrac12]$.
In the implementation, $\delta = 2\times10^{-2},\,2\times10^{-3},\,2\times10^{-4}$
correspond to noise levels $1\%,\,0.1\%,\,0.01\%$, respectively,
and $\delta=6\times10^{-2}$ corresponds to $3\%$ noise.
Since the reconstruction formula involves differentiated quantities, direct use of noisy terminal data may cause severe instability. Therefore, in the noisy experiments, 
we first smooth the terminal observation $g^\delta$. To this end, we employ a smoothing strategy related to \cite{chen2022} to denoise $g^\delta$ and obtain a stable approximation of $\Delta g$.
In all noisy-data figures, the reconstructions in the first row are shown after $10$ iterations,
while the second row displays the first three iterates under $3\%$ noise.
The reconstruction error at iteration $k$ is measured by the relative $L^2$-error
\[
\mathrm{RelErr}(k)
=
\frac{\|q^{(k)}-q_{\mathrm{true}}\|_{L^2(\Omega)}}
{\|q_{\mathrm{true}}\|_{L^2(\Omega)}}.
\]

\subsection{Smooth Example}

Figure~\ref{fig:smooth_clean} shows the noise-free results for the smooth example.
The reconstructed coefficient agrees very well with the exact one, and the relative error
decreases rapidly in the first few iterations.
This indicates that for a smooth drift field the proposed fixed-point iteration converges fast
and gives an accurate reconstruction.

Figure~\ref{fig:smooth_noise} shows the results for noisy terminal data.
The first row gives the reconstructions after $10$ iterations for noise levels
$1\%$, $0.1\%$, and $0.01\%$.
As the noise level decreases, the reconstructed profile becomes closer to the exact one,
and the level sets become smoother and more symmetric.
The second row shows the initial guess and the next two iterates for the case of $3\%$ noise.
Although the noise is relatively large, the global shape is already captured by the initial guess,
and the iterates improve step by step.
This agrees with the monotone behavior predicted by the theory.

\begin{figure}[H]
  \centering
  \subfigure[Exact coefficient]{%
      \includegraphics[width=0.32\textwidth]{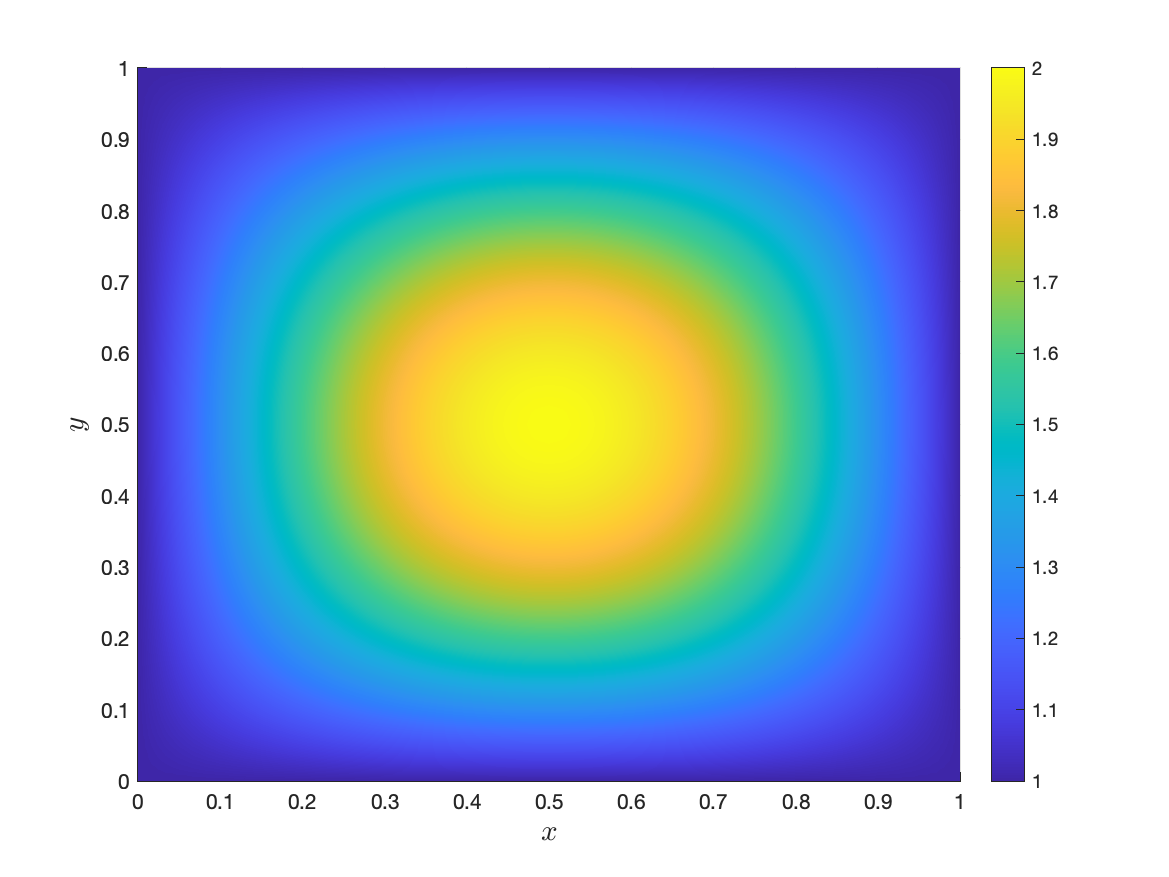}}
  \hfill
  \subfigure[Reconstructed coefficient]{%
      \includegraphics[width=0.32\textwidth]{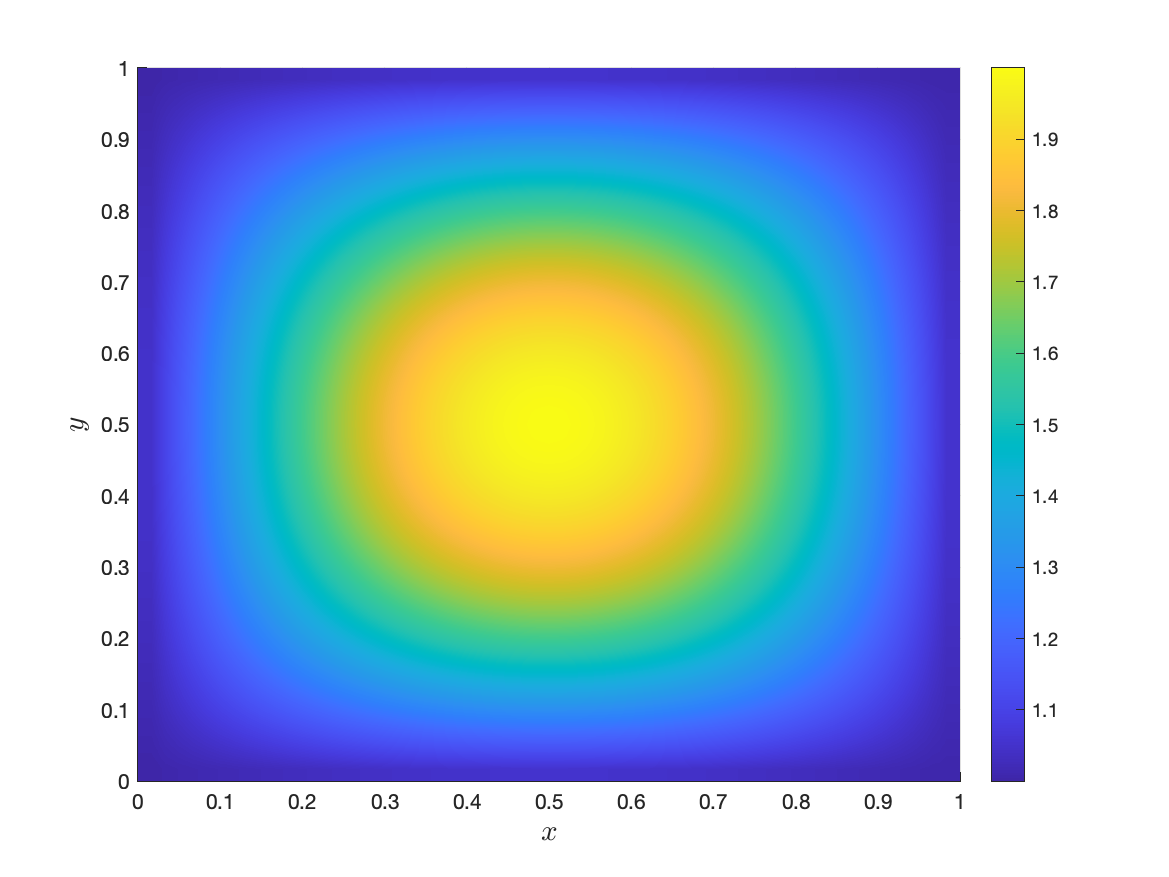}}
  \hfill
  \subfigure[Relative error]{%
      \includegraphics[width=0.32\textwidth]{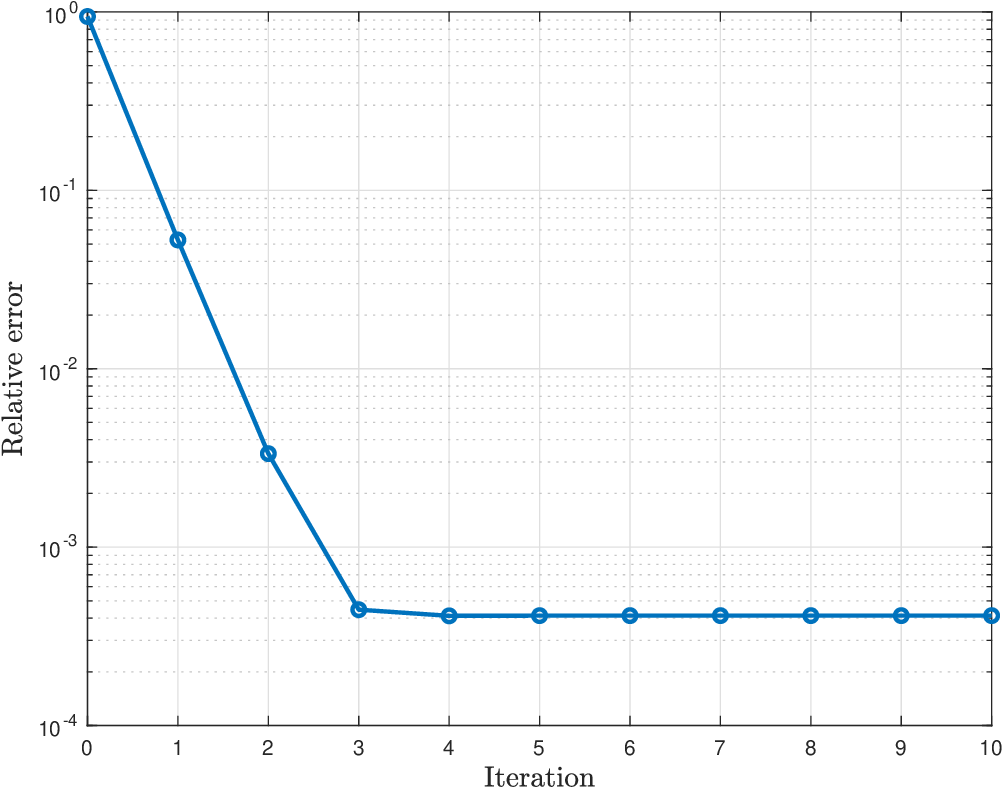}}
  \caption{Noise-free reconstruction for the smooth example.}
  \label{fig:smooth_clean}
\end{figure}

\begin{figure}[H]
  \centering
  \subfigure[$1\%$ noise, $k=10$]{%
      \includegraphics[width=0.32\textwidth]{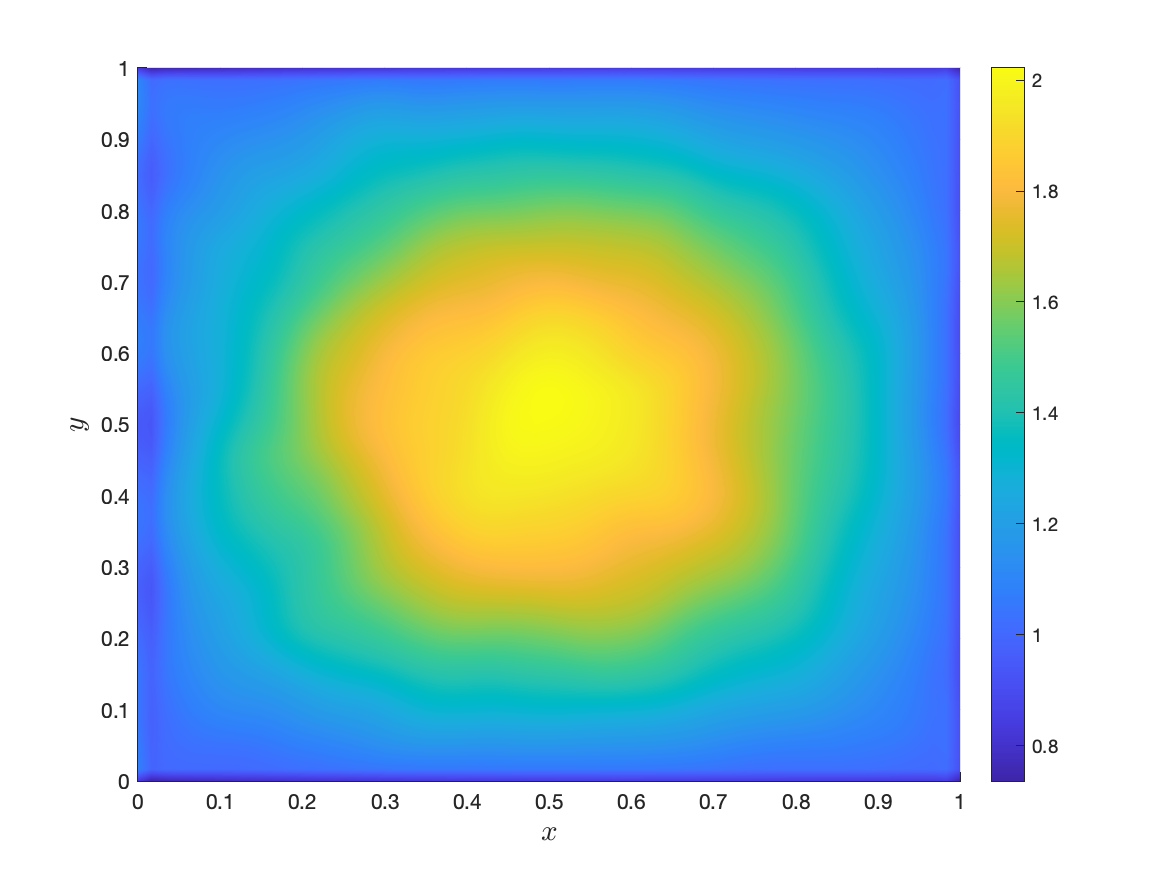}}
  \hfill
  \subfigure[$0.1\%$ noise, $k=10$]{%
      \includegraphics[width=0.32\textwidth]{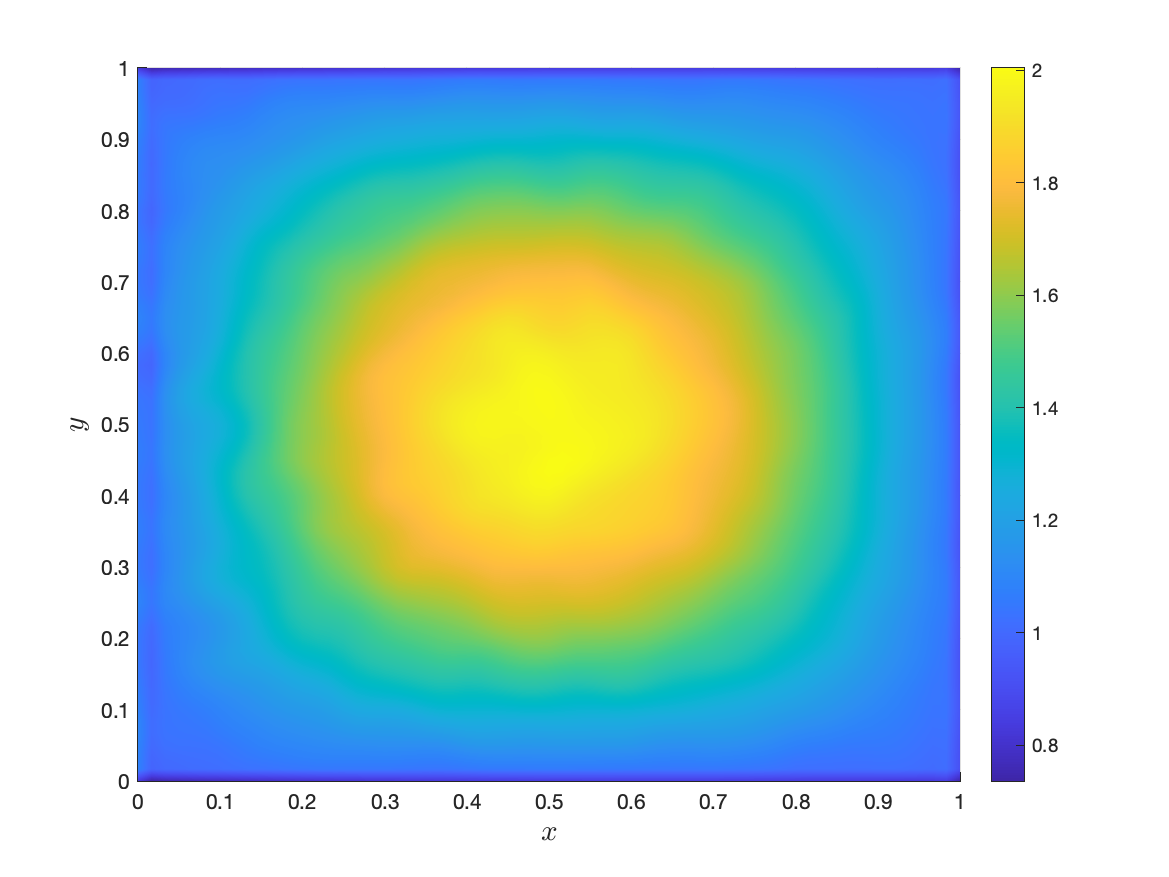}}
  \hfill
  \subfigure[$0.01\%$ noise, $k=10$]{%
      \includegraphics[width=0.32\textwidth]{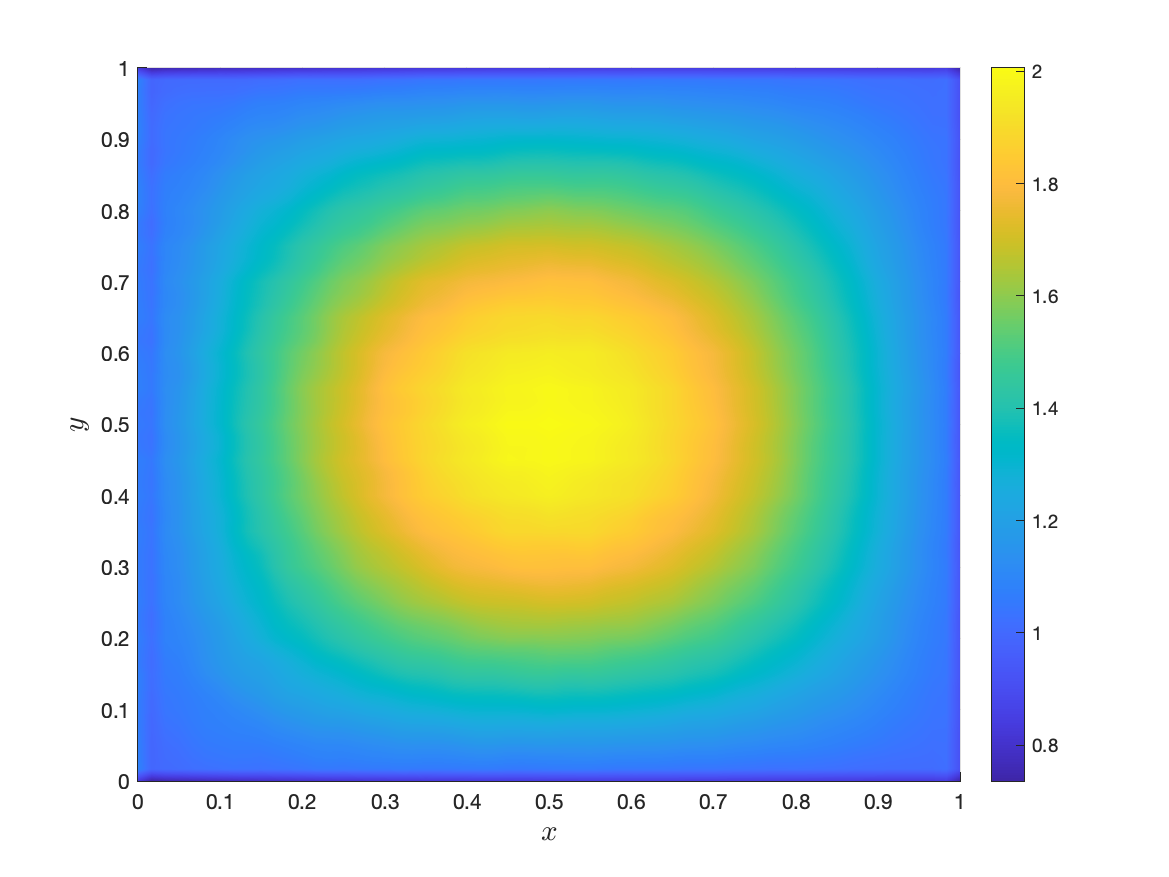}}
  \vspace{0.3cm}
  \subfigure[$3\%$ noise, $k=0$]{%
      \includegraphics[width=0.3\textwidth]{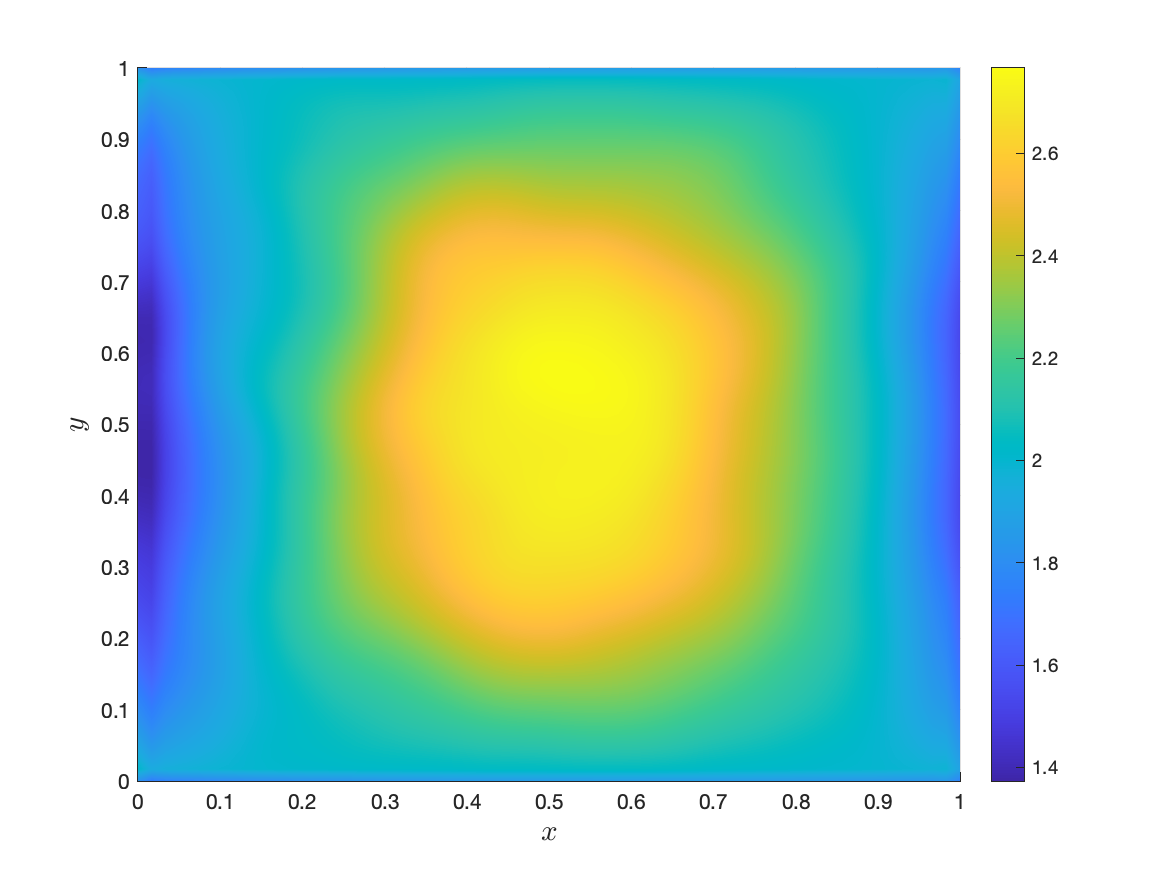}}
  \hfill
  \subfigure[$3\%$ noise, $k=1$]{%
      \includegraphics[width=0.3\textwidth]{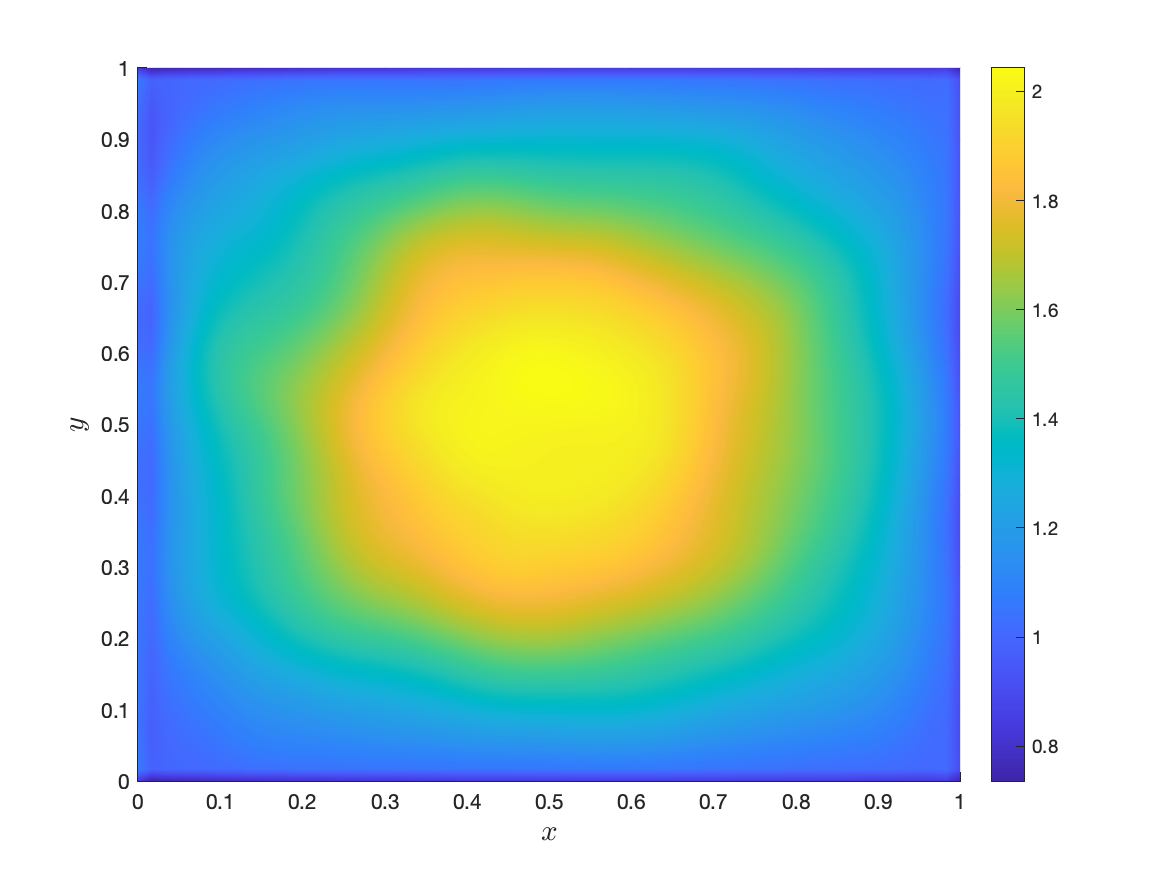}}
  \hfill
  \subfigure[$3\%$ noise, $k=2$]{%
      \includegraphics[width=0.3\textwidth]{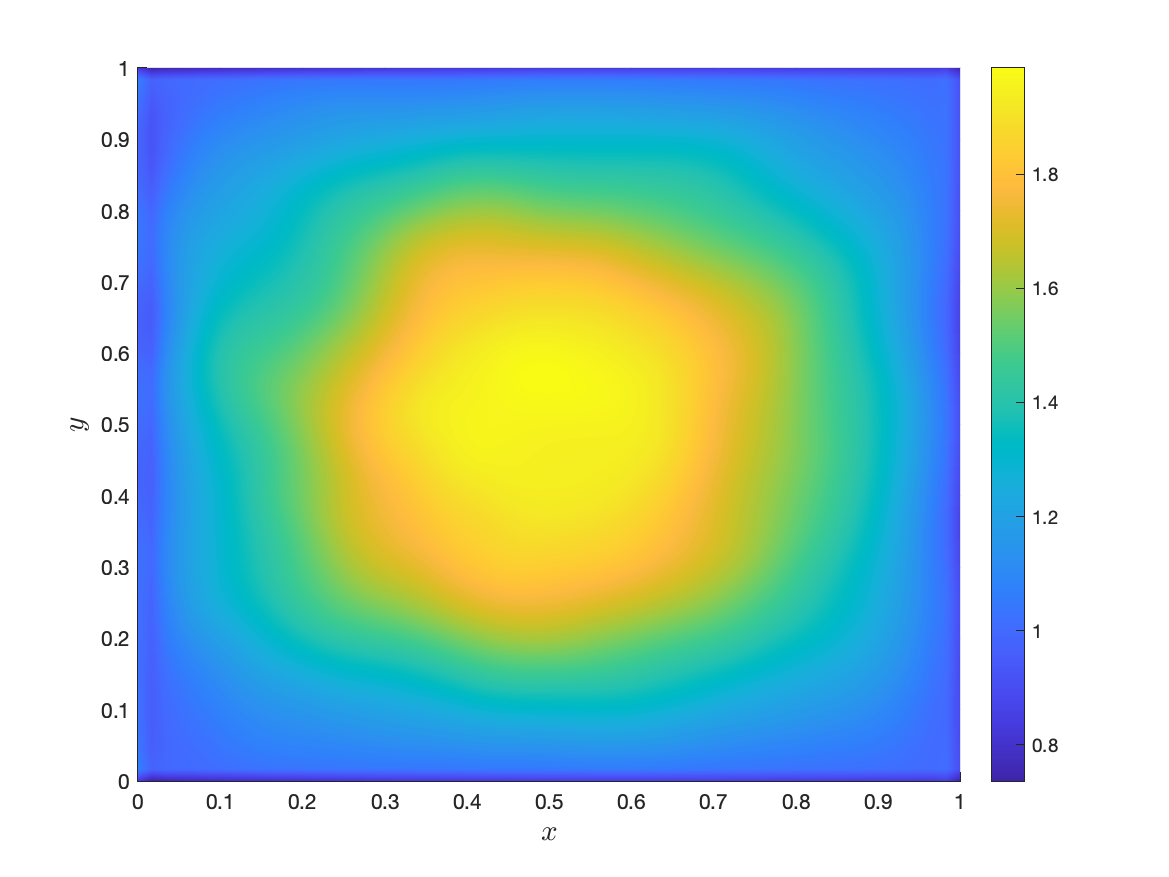}}
  \caption{Reconstructions for the smooth example with noisy data. The first row shows the results after $10$ iterations for three noise levels; the second row shows the first three iterates for $3\%$ noise.}
  \label{fig:smooth_noise}
\end{figure}

\subsection{Piecewise Constant Example}

Figure~\ref{fig:pc_clean} presents the noise-free results for the square example.
The method correctly recovers the location and contrast of the inclusion,
and the relative error decreases monotonically with the iteration number.
Although the exact coefficient is discontinuous, the reconstructed interface remains sharp,
with only mild smoothing near the corners caused by the coarse inversion grid and the discrete solver.

Figure~\ref{fig:pc_noise} reports the reconstructions for noisy data.
The first row shows the results after $10$ iterations for noise levels
$1\%$, $0.1\%$, and $0.01\%$.
The support of the square inclusion is clearly identified in all three cases,
and the reconstruction becomes more accurate as the noise level decreases.
The second row displays the first three iterates for $3\%$ noise.
One can see that the background distortion is reduced from one step to the next,
while the shape of the square becomes clearer.
This provides a direct numerical illustration of the monotone improvement of the iteration.

\begin{figure}[H]
  \centering
  \subfigure[Exact coefficient]{%
    \includegraphics[width=0.32\textwidth]{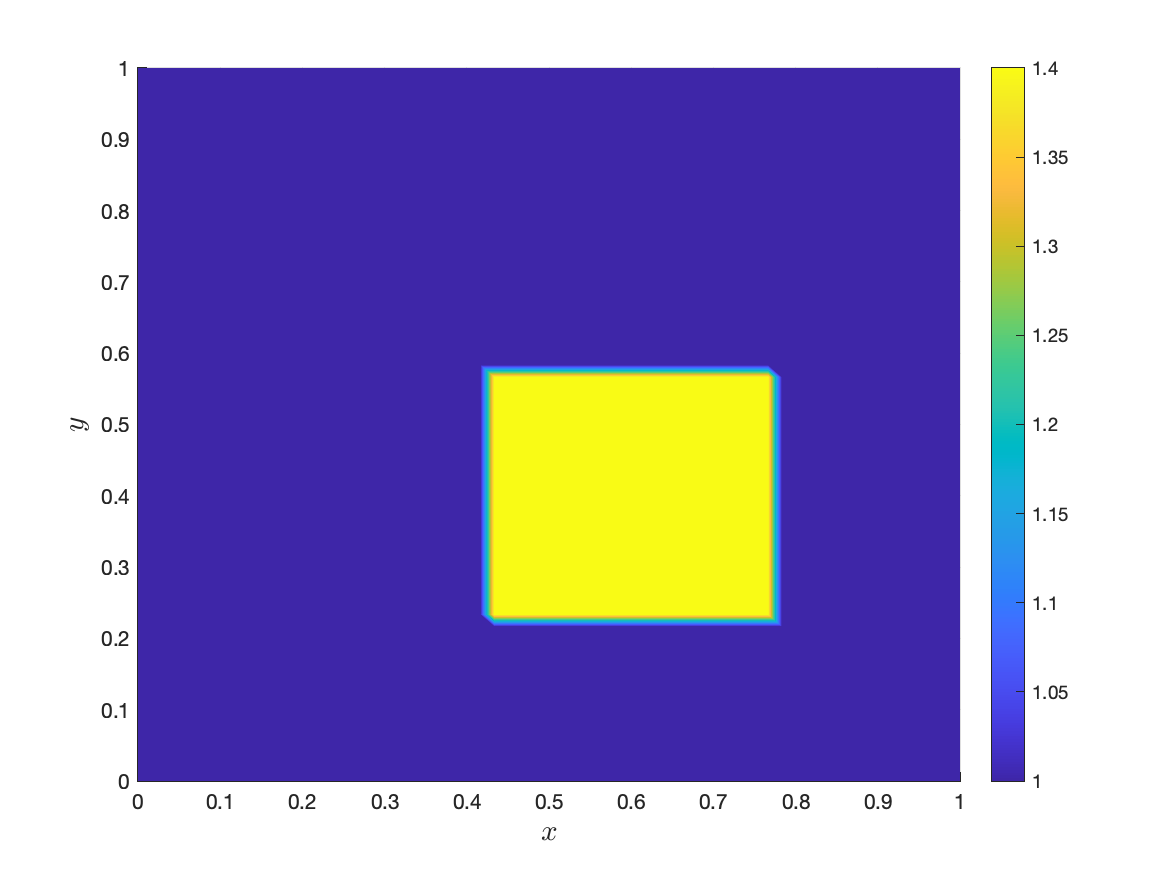}}
  \hfill
  \subfigure[Reconstructed coefficient]{%
    \includegraphics[width=0.32\textwidth]{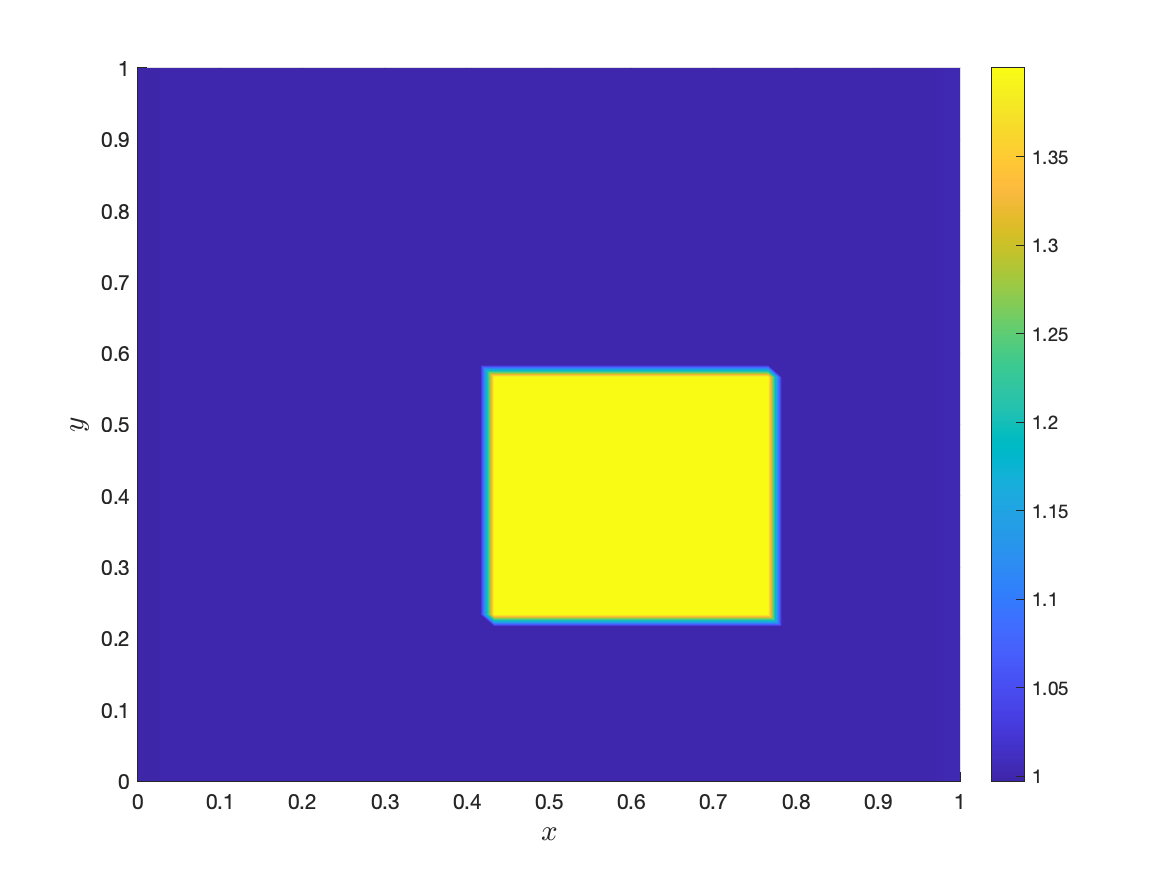}}
  \hfill
  \subfigure[Relative error]{%
    \includegraphics[width=0.32\textwidth]{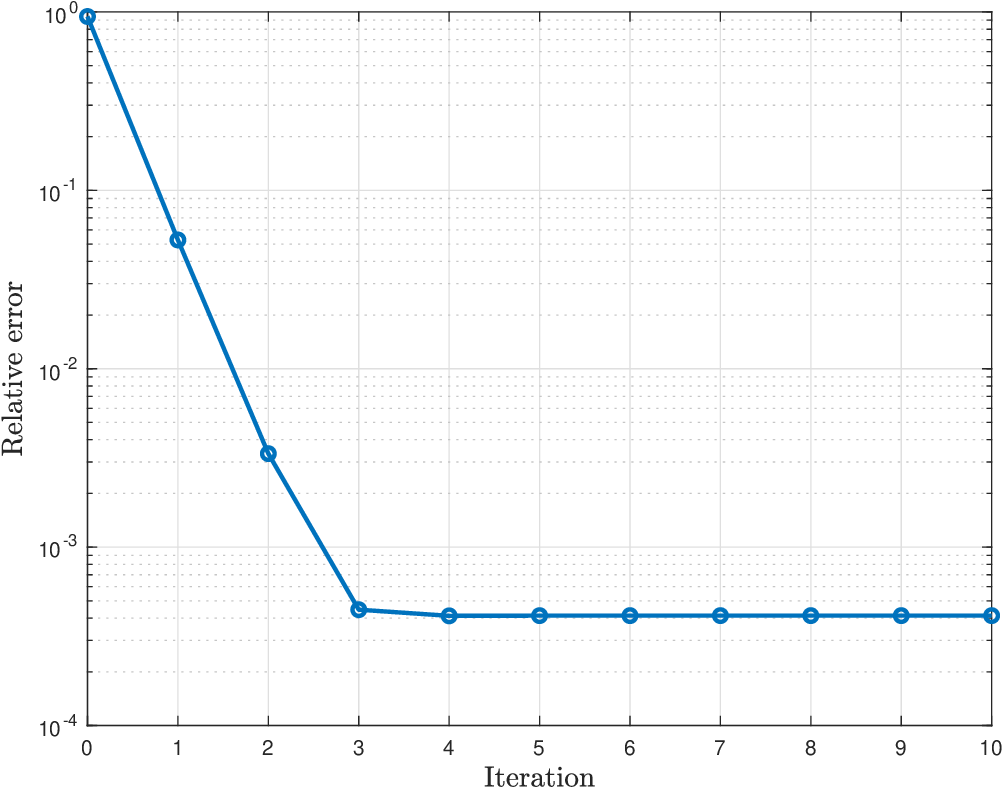}}
    \caption{Noise-free reconstruction for the piecewise constant example.}
    \label{fig:pc_clean}
\end{figure}

\begin{figure}[H]
  \centering
  \subfigure[$1\%$ noise, $k=10$]{%
      \includegraphics[width=0.32\textwidth]{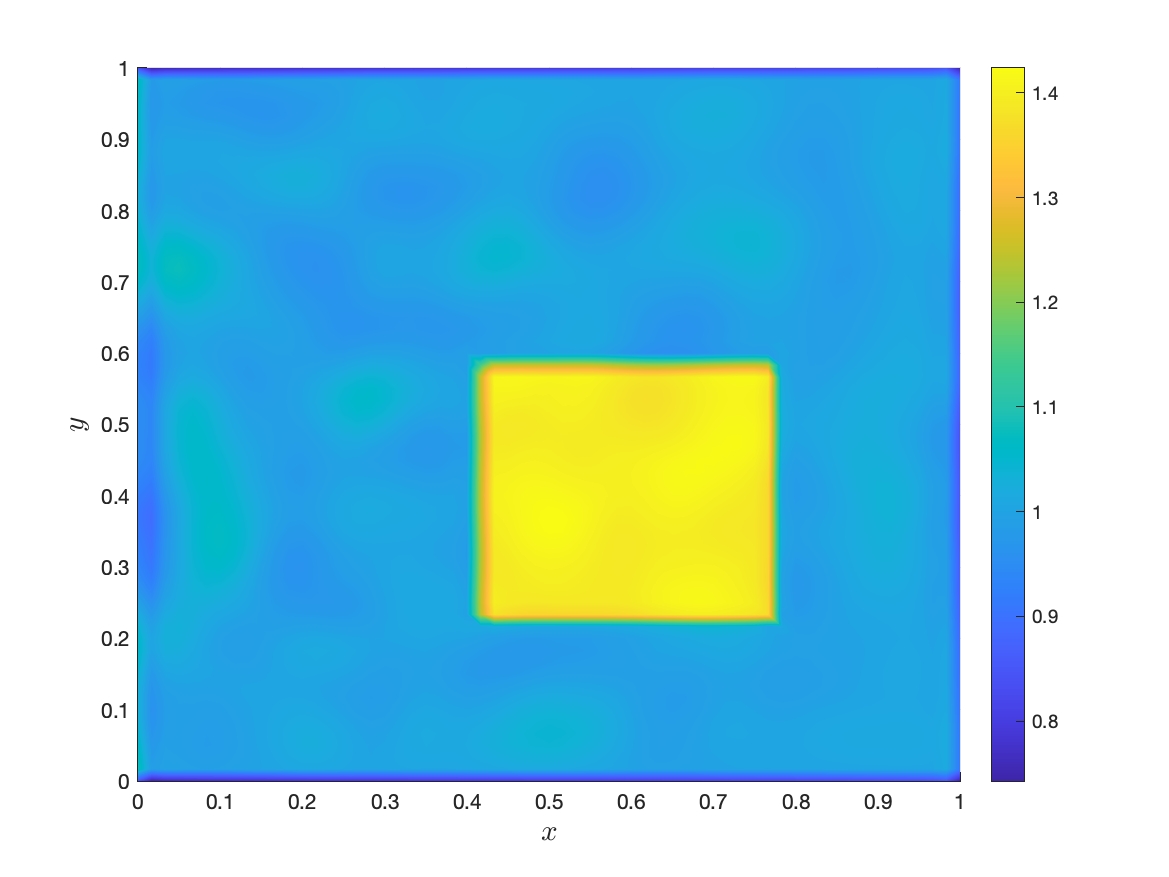}}
  \hfill
  \subfigure[$0.1\%$ noise, $k=10$]{%
      \includegraphics[width=0.32\textwidth]{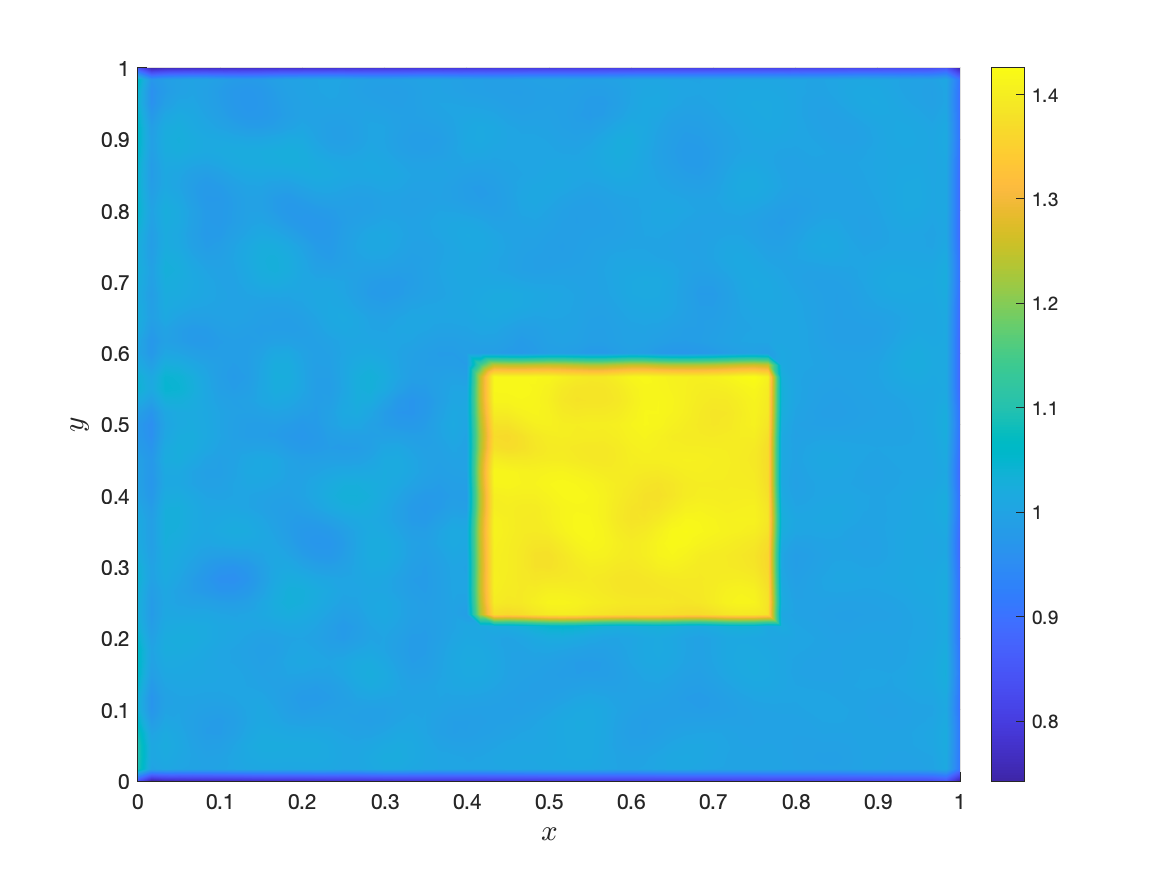}}
  \hfill
  \subfigure[$0.01\%$ noise, $k=10$]{%
      \includegraphics[width=0.32\textwidth]{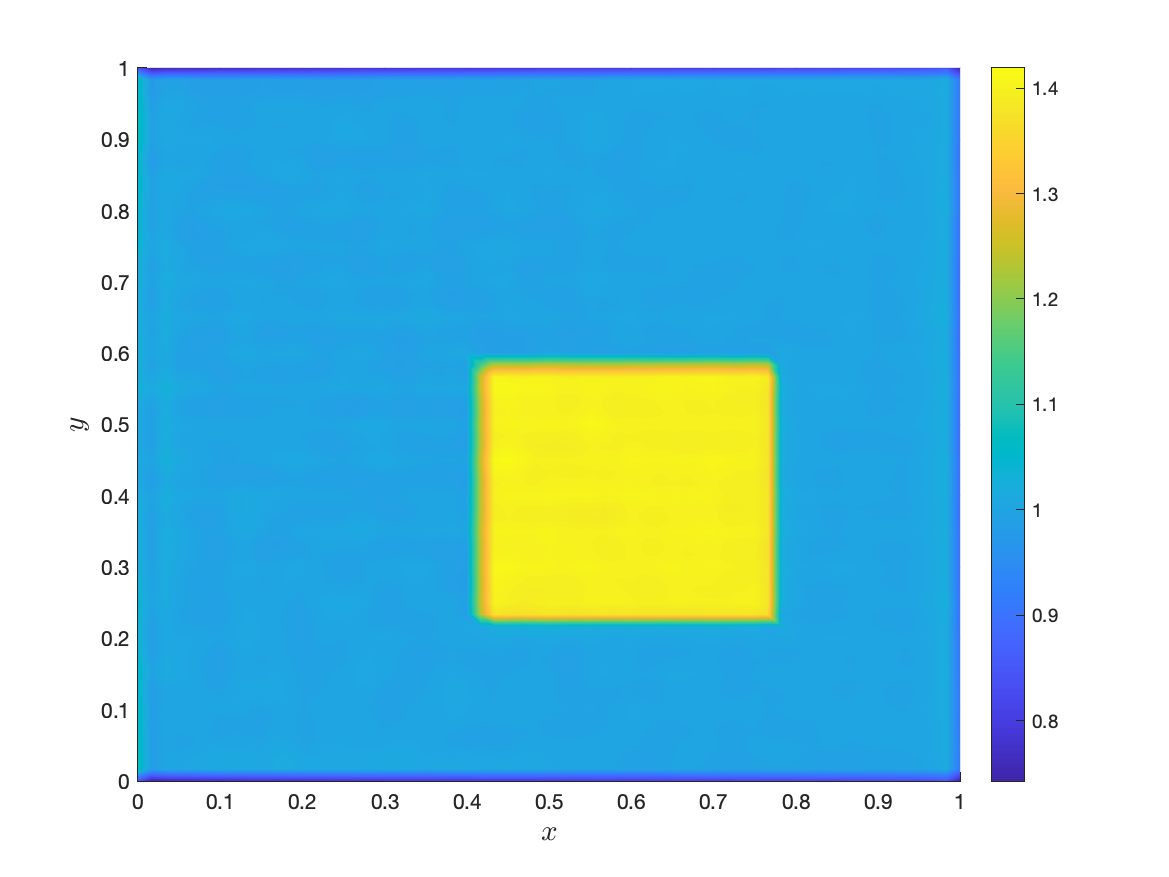}}
      \vspace{0.3cm}
  \subfigure[$3\%$ noise, $k=0$]{%
      \includegraphics[width=0.3\textwidth]{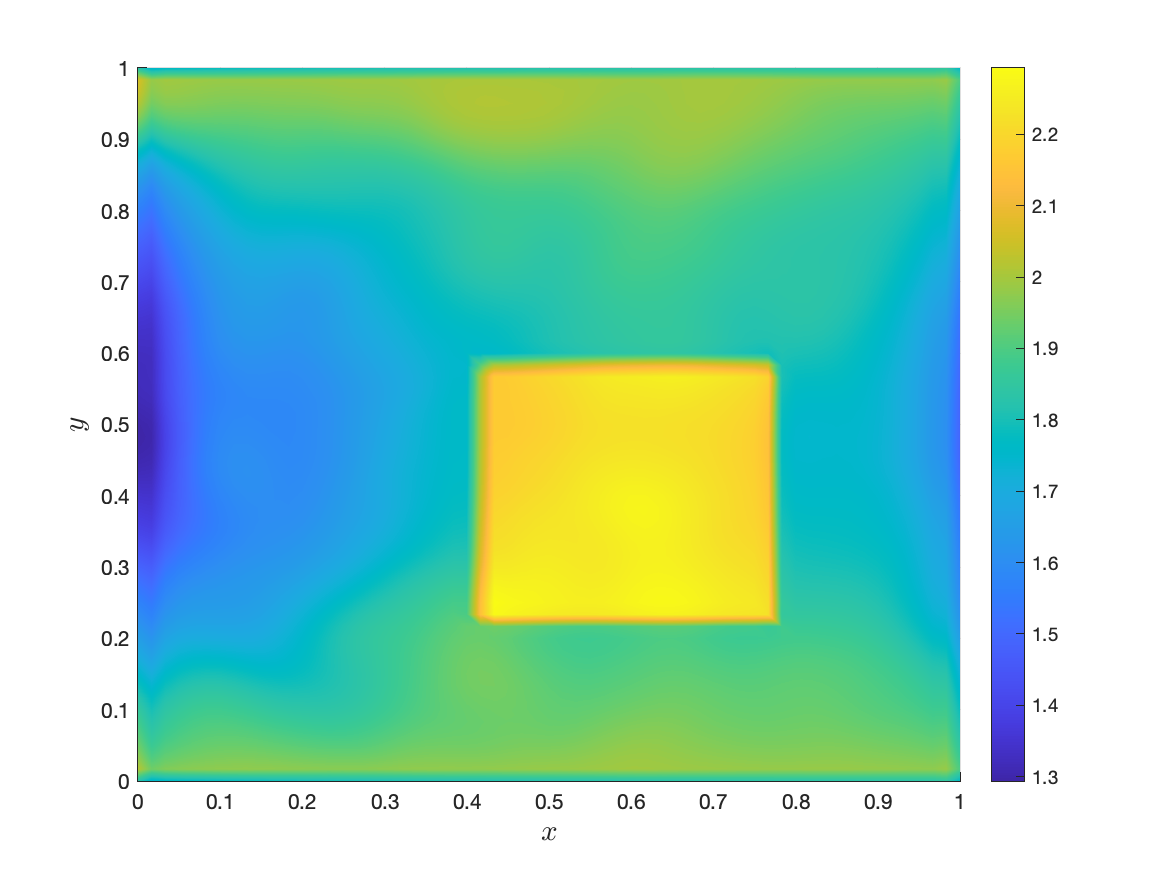}}
  \hfill
  \subfigure[$3\%$ noise, $k=1$]{%
      \includegraphics[width=0.3\textwidth]{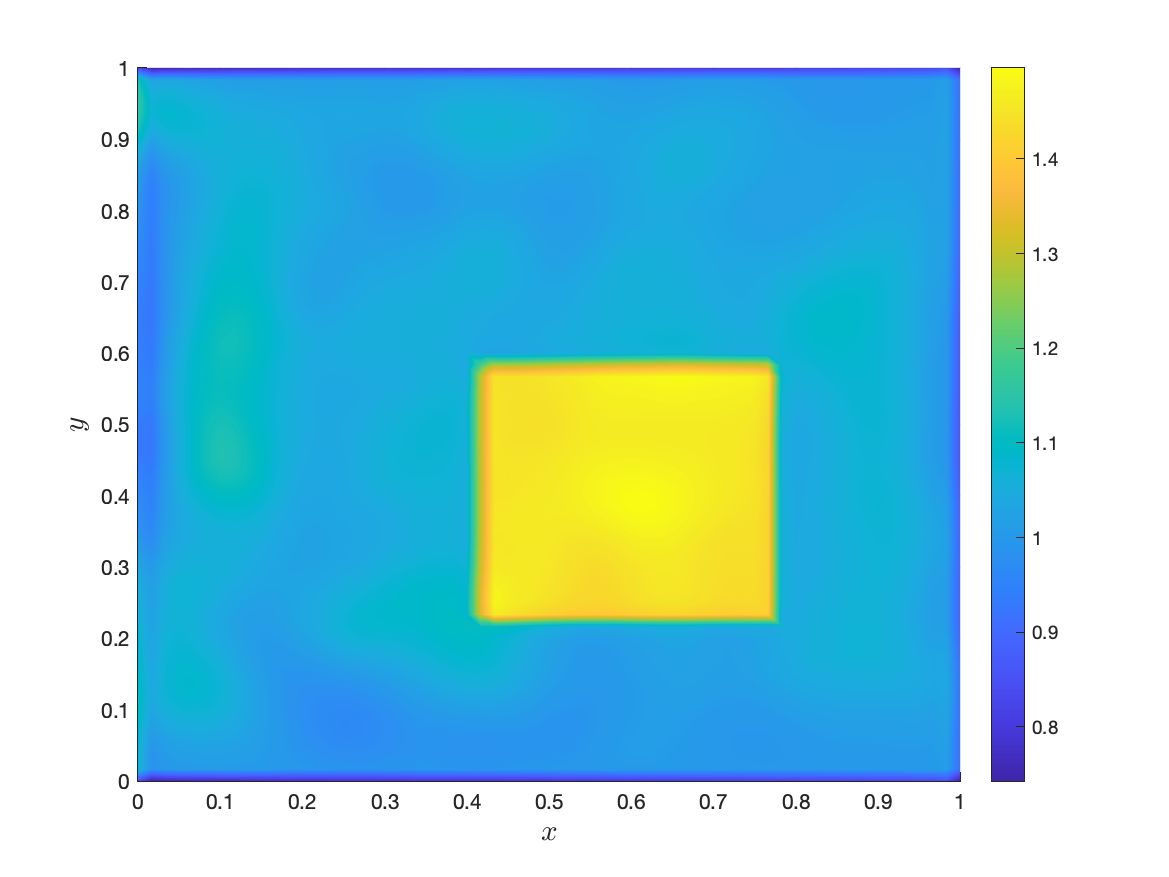}}
  \hfill
  \subfigure[$3\%$ noise, $k=2$]{%
      \includegraphics[width=0.3\textwidth]{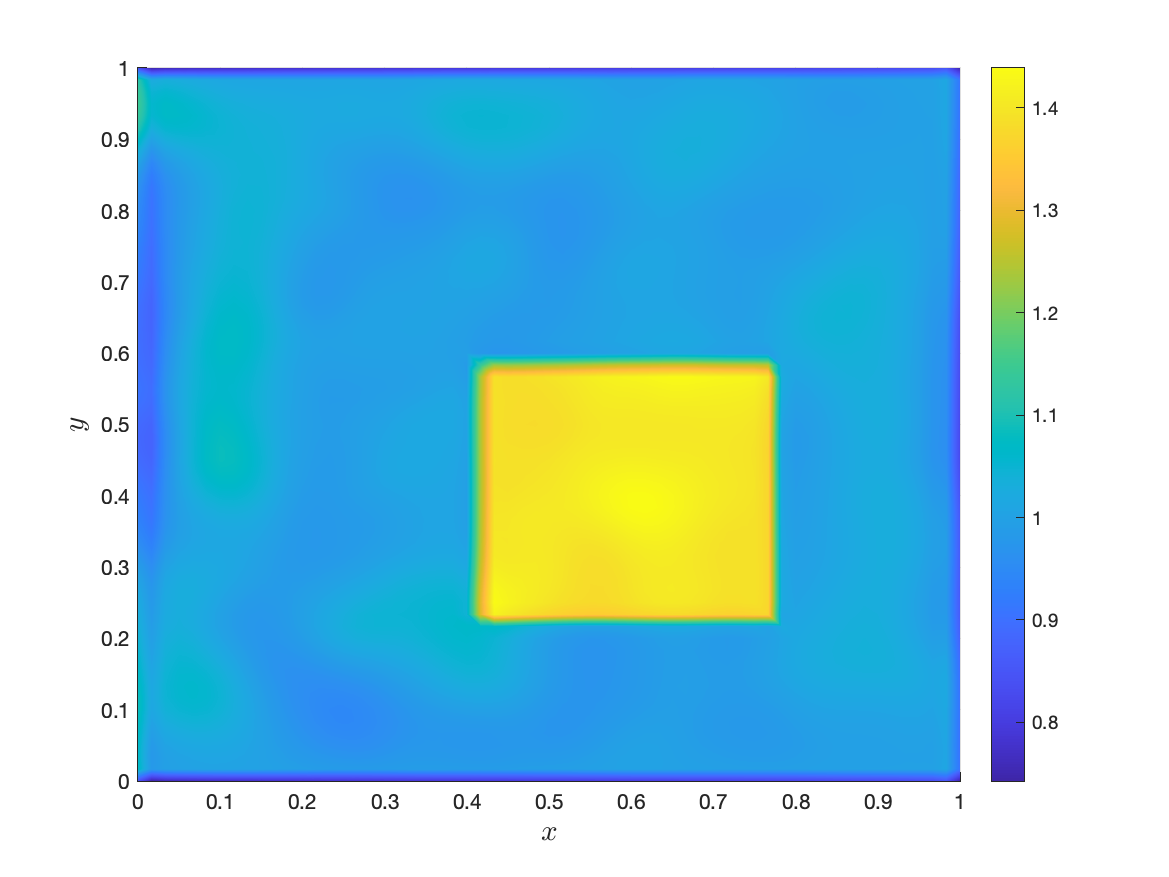}}
      \caption{Reconstructions for the piecewise constant example with noisy data. The first row shows the results after $10$ iterations for three noise levels; the second row shows the first three iterates for $3\%$ noise.}
      \label{fig:pc_noise}
\end{figure}

\subsection{Chinese Character Example}

The third example is more complicated because the exact coefficient contains several corners,
thin structures, and two inner holes.
Figure~\ref{fig:char_clean} shows the noise-free reconstruction.
The main geometry of the target is well recovered, including the two inner holes
and the central vertical bar.
The relative error decreases rapidly in the first few iterations and then becomes nearly flat.

Figure~\ref{fig:char_noise} shows the results for noisy terminal data.
The first row gives the reconstructions after $10$ iterations for noise levels
$1\%$, $0.1\%$, and $0.01\%$.
Even in the presence of noise, the main topological features remain visible,
and the reconstruction improves as the noise level decreases.
The second row shows the first three iterates for $3\%$ noise.
Although some background variation is still visible, the main character-shaped structure is already well captured at the initial step.
As the iteration proceeds from $k=0$ to $k=2$, the background becomes more uniform and the coefficient is recovered more clearly, showing a step-by-step improvement of the reconstruction.
This suggests that the proposed method remains stable even for nonsmooth coefficients
with more complicated geometry.

\begin{figure}[H]
  \centering
  \subfigure[Exact coefficient]{%
    \includegraphics[width=0.32\textwidth]{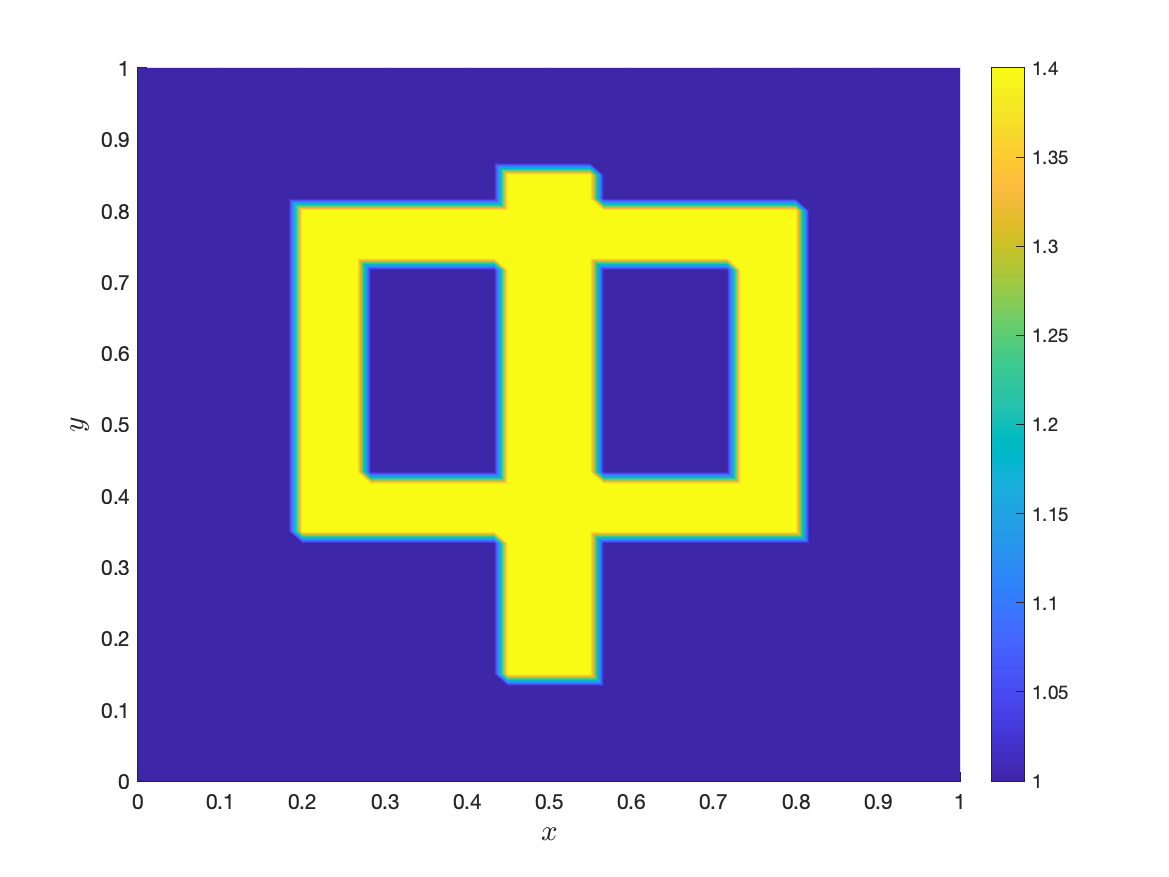}}
  \hfill
  \subfigure[Reconstructed coefficient]{%
    \includegraphics[width=0.32\textwidth]{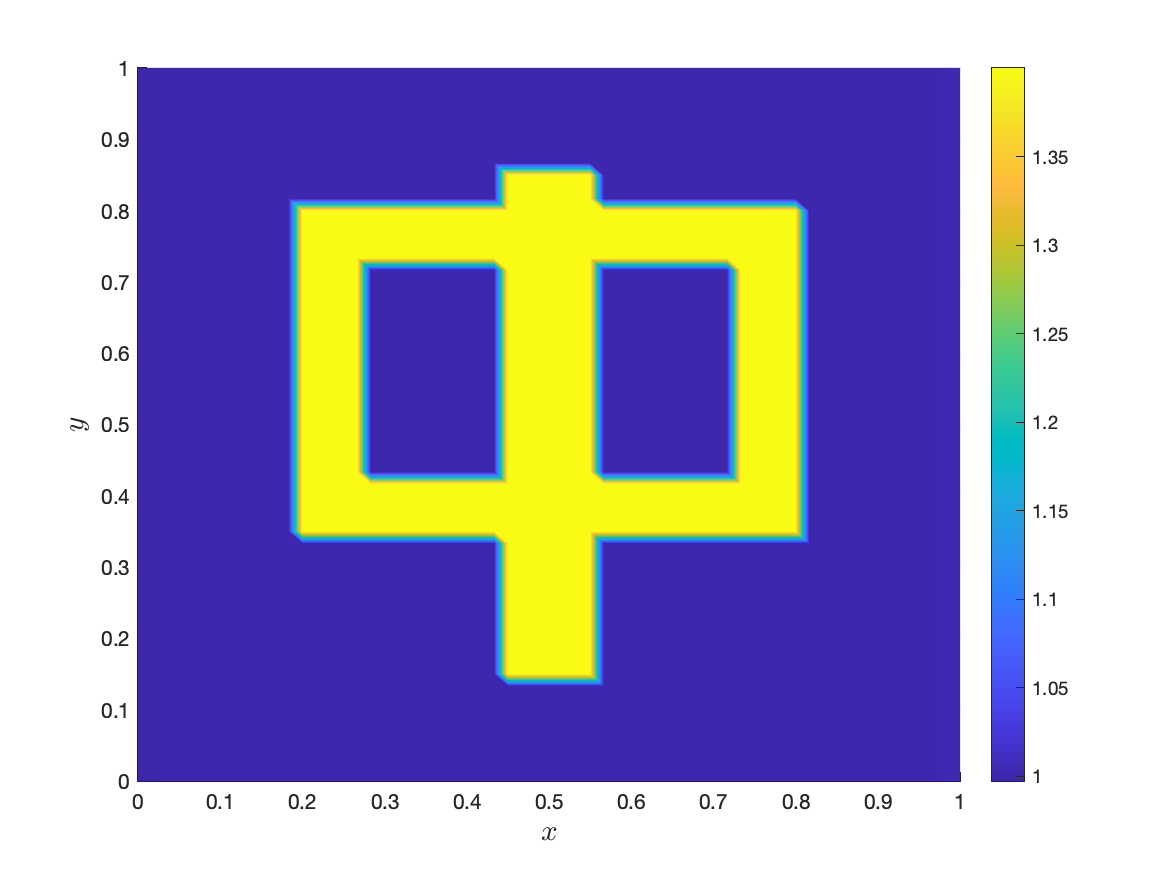}}
  \hfill
  \subfigure[Relative error]{%
    \includegraphics[width=0.32\textwidth]{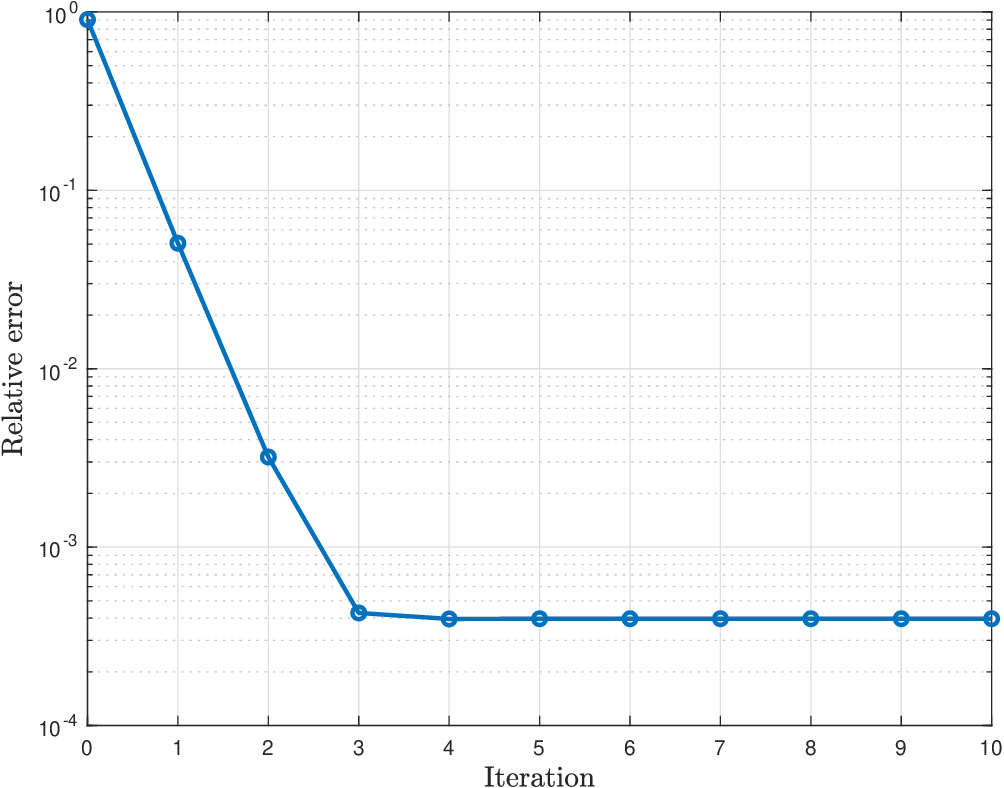}}
  \caption{Noise-free reconstruction for the Chinese character example.}
    \label{fig:char_clean}
\end{figure}

\begin{figure}[H]
  \centering
  \subfigure[$1\%$ noise, $k=10$]{%
      \includegraphics[width=0.32\textwidth]{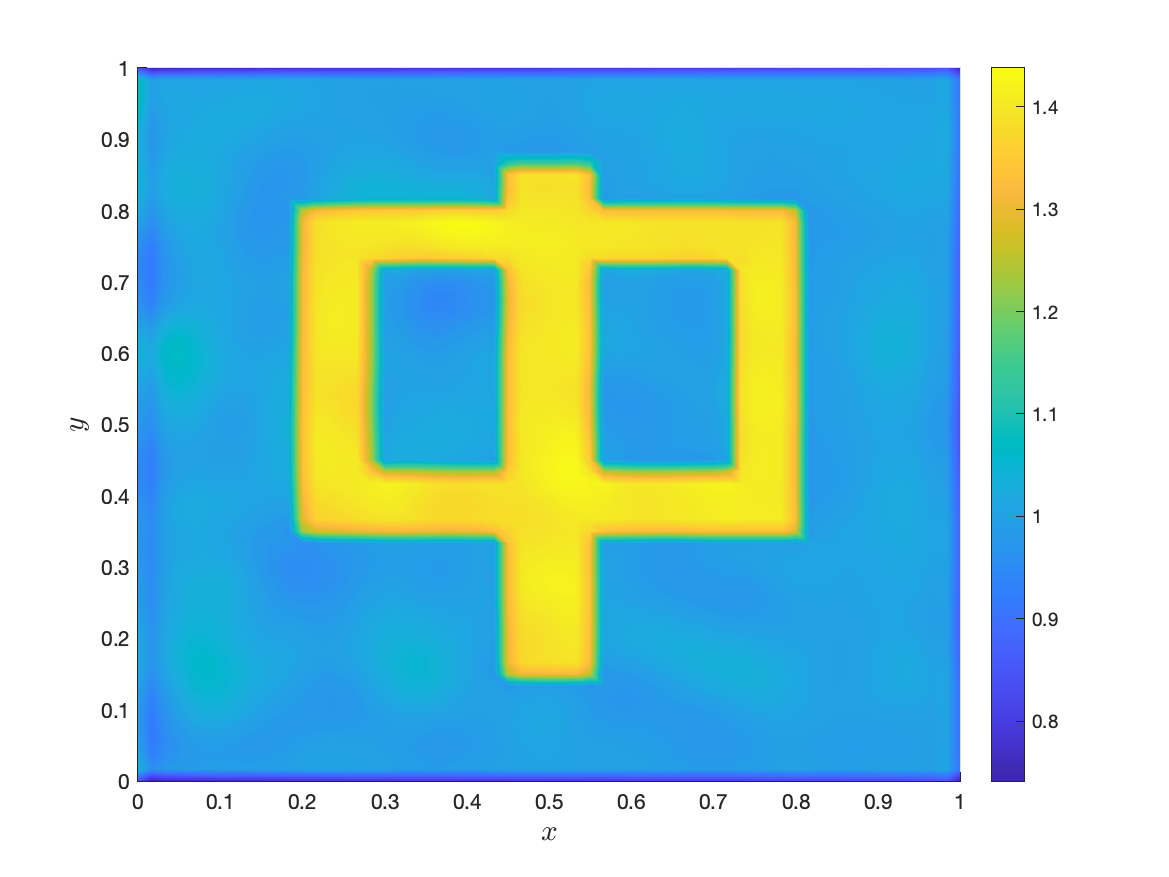}}
  \hfill
  \subfigure[$0.1\%$ noise, $k=10$]{%
      \includegraphics[width=0.32\textwidth]{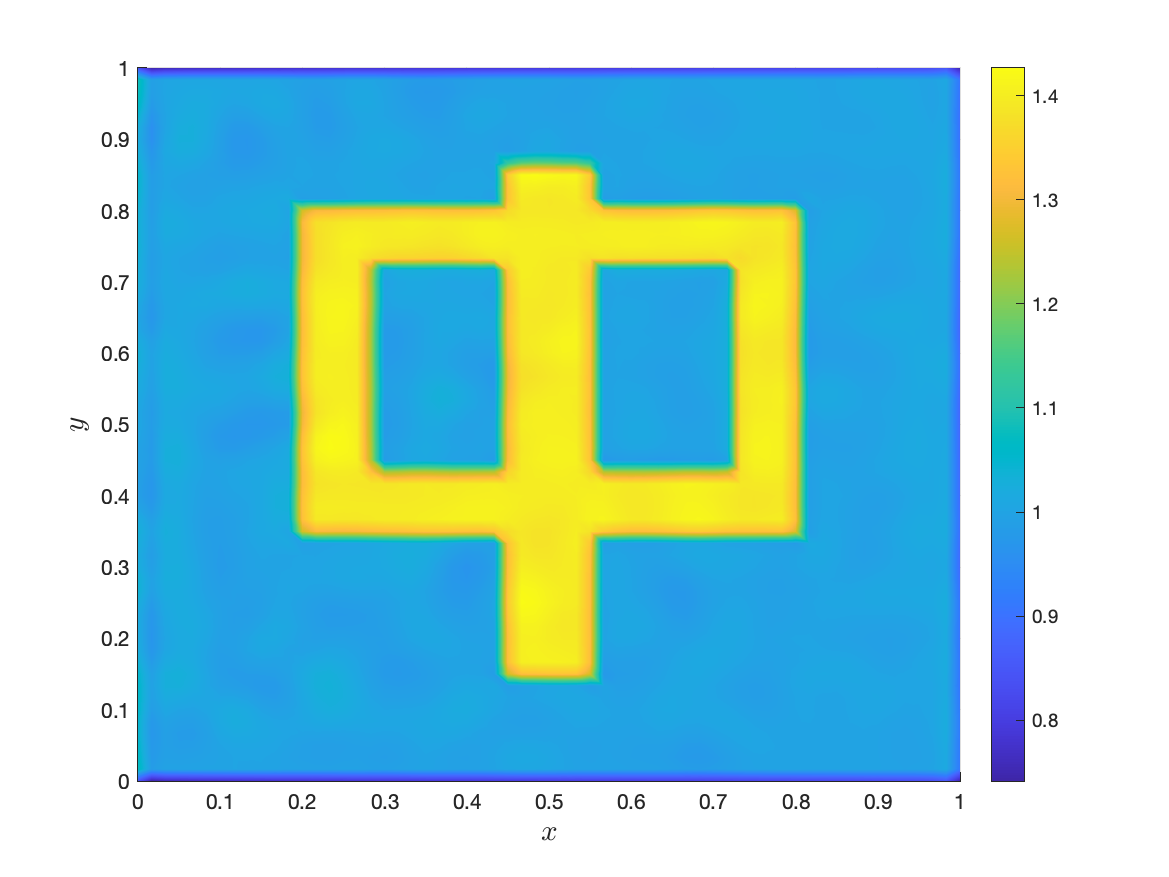}}
  \hfill
  \subfigure[$0.01\%$ noise, $k=10$]{%
      \includegraphics[width=0.32\textwidth]{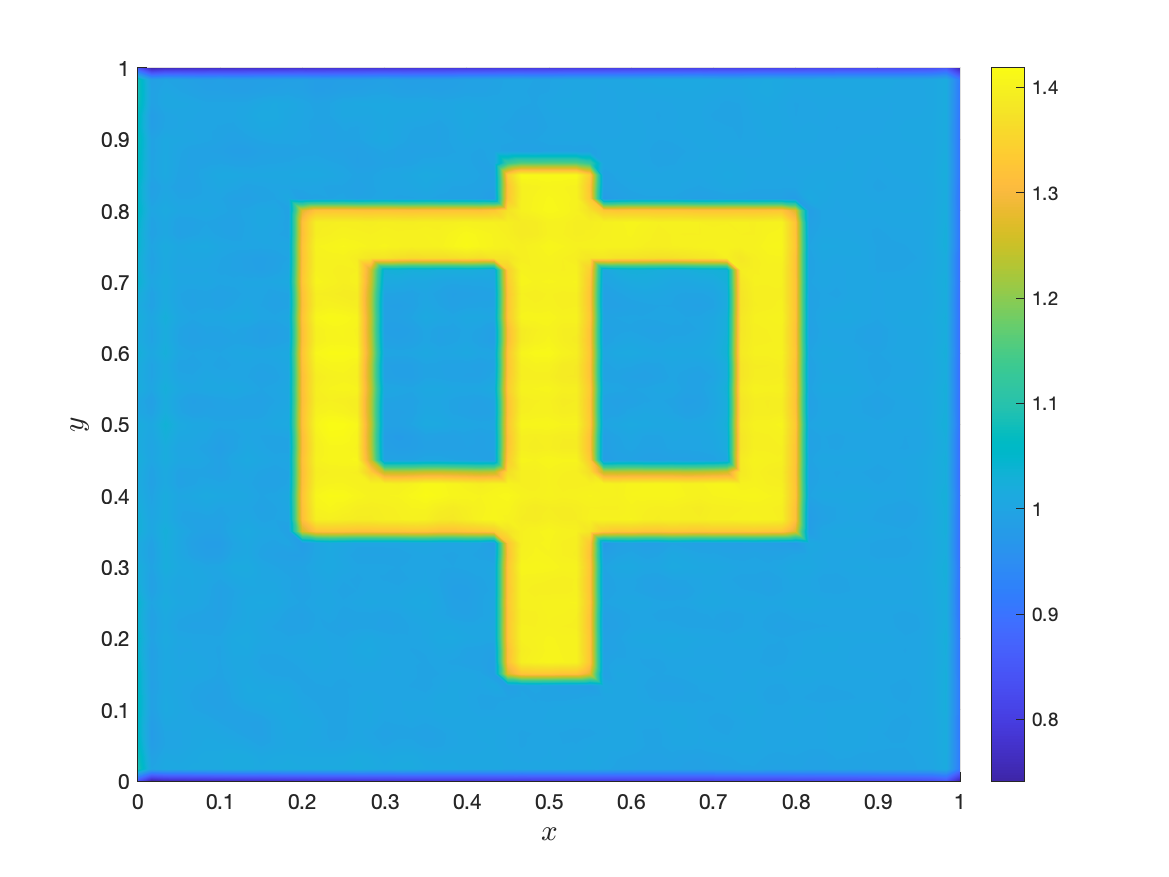}}
      \vspace{0.3cm}
  \subfigure[$3\%$ noise, $k=0$]{%
      \includegraphics[width=0.3\textwidth]{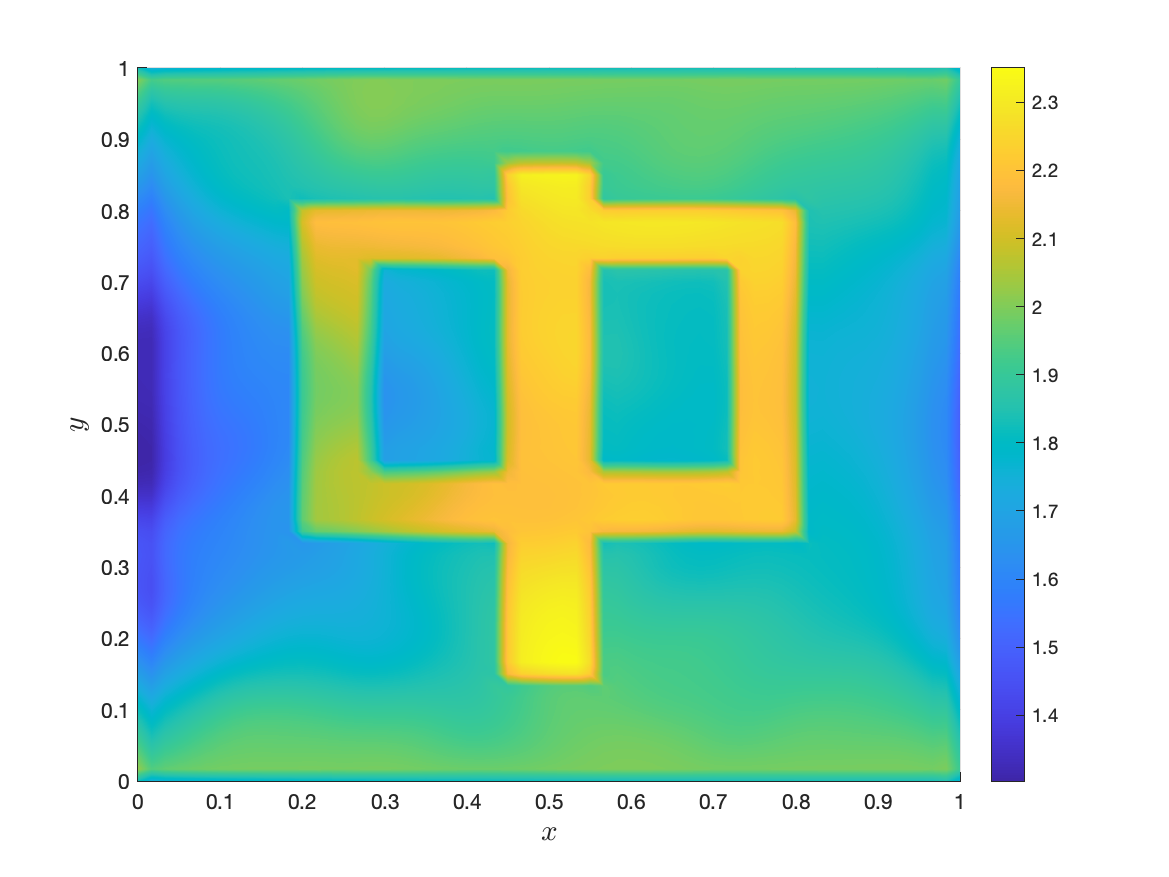}}
  \hfill
  \subfigure[$3\%$ noise, $k=1$]{%
      \includegraphics[width=0.3\textwidth]{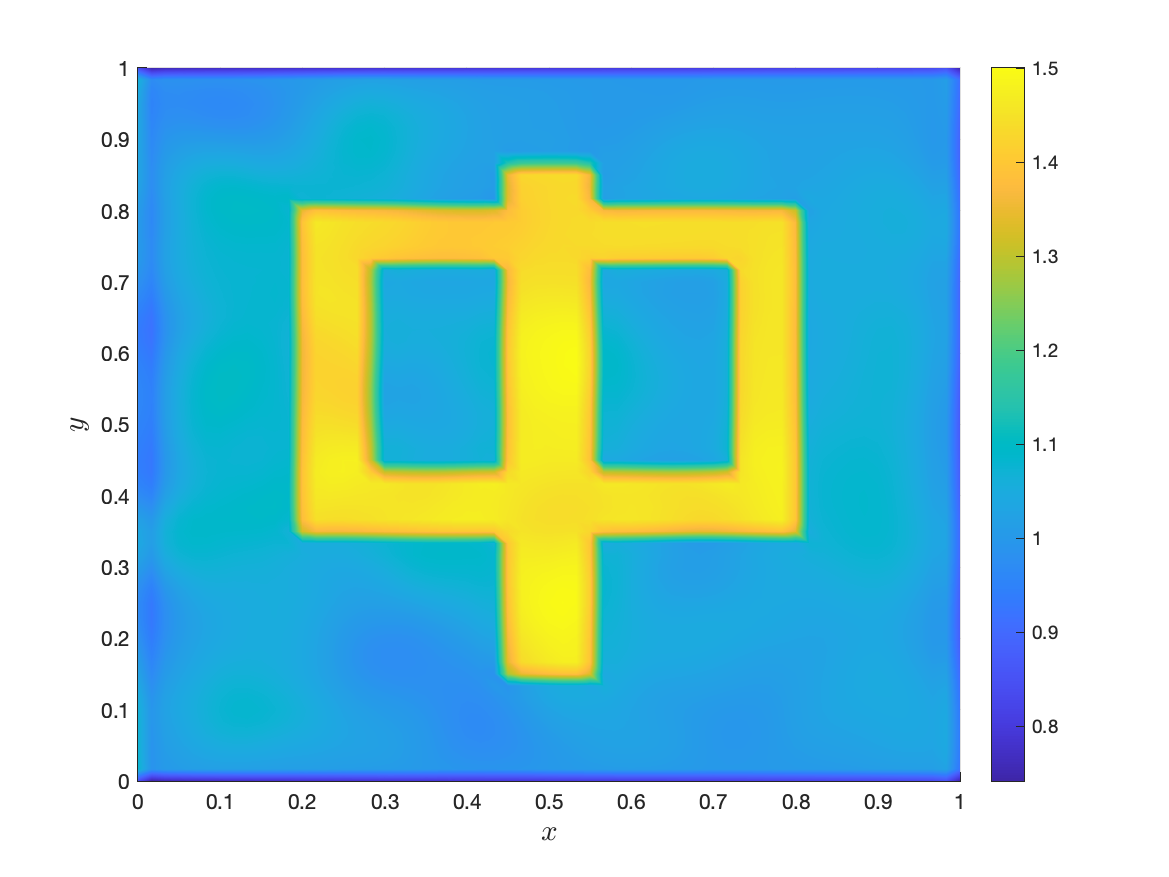}}
  \hfill
  \subfigure[$3\%$ noise, $k=2$]{%
      \includegraphics[width=0.3\textwidth]{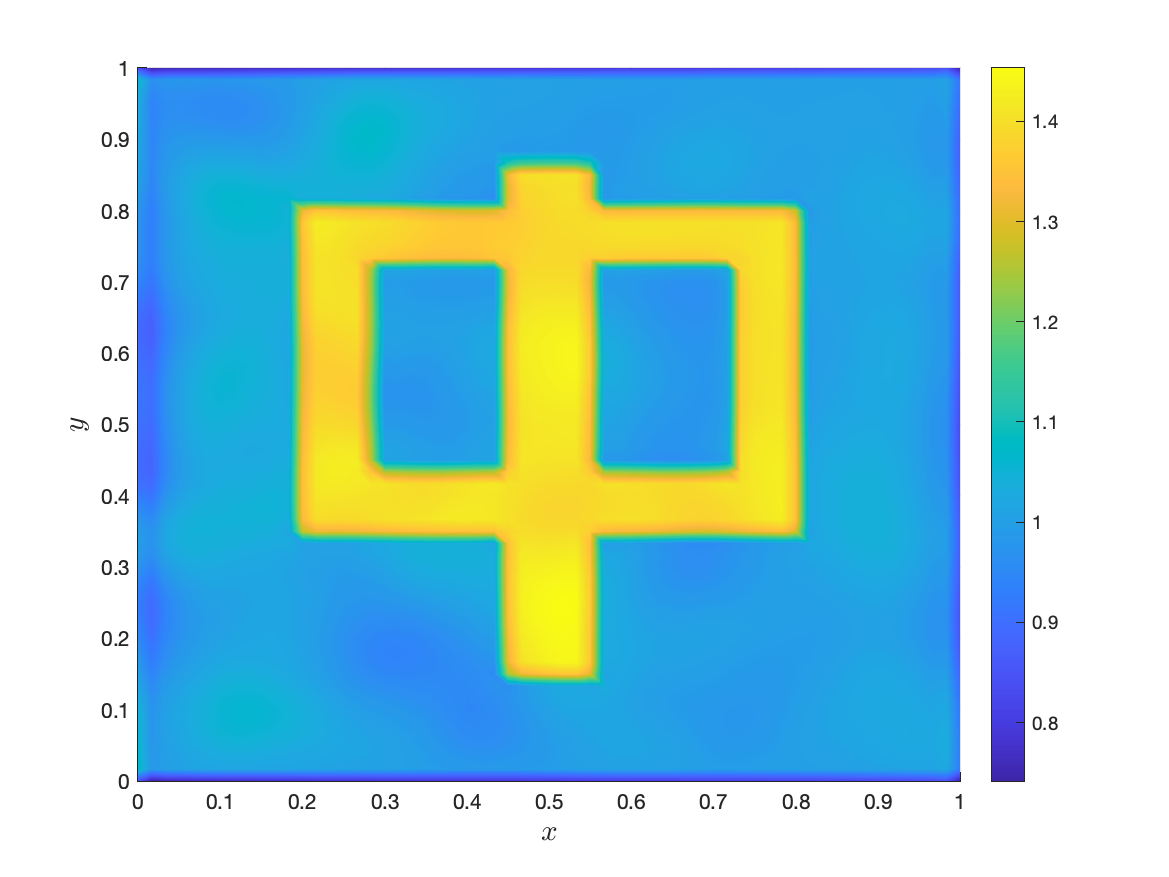}}
  \caption{Reconstructions for the Chinese character example with noisy data. The first row shows the results after $10$ iterations for three noise levels; the second row shows the first three iterates for $3\%$ noise.}
  \label{fig:char_noise}
\end{figure}

\section*{Acknowledgments.}
Zhidong Zhang is supported by the National Key Research and Development Plan of China (Grant No. 2023YFB3002400). Wenlong Zhang is partially supported by the National Natural Science Foundation of China under grant numbers No.12371423, No.12241104 and No.12561160122.

\bibliographystyle{abbrv}
\bibliography{ref}

\end{document}